\documentclass[a4paper, 12pt]{amsart}
\usepackage[margin=0.8in]{geometry}
\usepackage[utf8]{inputenc}
\usepackage[T1]{fontenc}
\usepackage{multirow}
\usepackage{booktabs}
\usepackage{comment}
\usepackage{amssymb, amsmath}
\usepackage{amsfonts}
\usepackage{enumitem}
\usepackage{xcolor}
\usepackage{float}
\usepackage{mathtools}
\usepackage{tikz}
\usepackage{caption}
\usepackage{array}
\usepackage{colortbl}
\usepackage{subcaption}
\usepackage{pgfplotstable}
\usepackage{pgfplots}
\graphicspath{{figures/}}
\pgfplotsset{compat=newest}
\usetikzlibrary{arrows}
\usetikzlibrary{decorations.text}
\usetikzlibrary{decorations.markings}
\pgfplotsset{compat=1.17}
\usetikzlibrary{arrows.meta, calc}
\usepackage[protrusion=true,expansion=false]{microtype}
\usepackage[hidelinks]{hyperref}

\newcommand{\abs}[1]{\lvert #1 \rvert}

\renewcommand{\Re}{\operatorname{Re}}
\renewcommand{\Im}{\operatorname{Im}}
\newcommand{\tr}{\operatorname{tr}}
\renewcommand{\det}{\operatorname{det}}
\newcommand{\zetaR}{\zeta_{\scriptscriptstyle R}}
\newcommand{\xiR}{\xi_{\scriptscriptstyle R}}
\newtheorem{theorem}{Theorem}[section]
\newtheorem{proposition}[theorem]{Proposition}
\newtheorem{lemma}[theorem]{Lemma}

\newtheorem{definition}[theorem]{Definition}
\newtheorem{assumption}[theorem]{Assumption}

\newtheorem{remark}[theorem]{Remark}

\definecolor{stablegreen}{RGB}{30, 140, 60}
\definecolor{unstablered}{RGB}{200, 40, 40}
\definecolor{saddleorange}{RGB}{210, 120, 0}
\definecolor{regionblue}{RGB}{220, 235, 255}
\definecolor{regionred}{RGB}{255, 225, 220}
\definecolor{twblue}{RGB}{20, 80, 180}
\definecolor{swred}{RGB}{180, 20, 20}

\title[Wave Selection at an $O(2)$-Hopf Bifurcation]{Wave Selection at an $O(2)$-Hopf Bifurcation in Conservative Two-Component PDE Systems}

\author[Özer]{Saadet S. Özer$^1$}
\address[Özer$^1$]{Department of Mathematics, Istanbul Technical University, 34467 Istanbul, Turkey}
\email{ozersaa@itu.edu.tr}

\author[Şengül]{Taylan Şengül$^2$}
\address[Şengül$^2$]{Department of Mathematics, Marmara University, 34722 Istanbul, Turkey}
\email{taylan.sengul@marmara.edu.tr}

\author[Tiryakioglu]{Burhan Tiryakioglu$^3$}
\address[Tiryakioglu$^3$]{Department of Mathematics, Marmara University, 34722 Istanbul, Turkey}
\email{burhan.tiryakioglu@marmara.edu.tr}
\date{\today}

\usepackage{placeins}
\begin{document}

\begin{abstract}
At an $O(2)$-equivariant Hopf bifurcation a spatially periodic system selects between traveling waves and standing waves, and which branch appears and whether it is stable is decided by two cubic normal form coefficients.
Obtaining these coefficients for a given PDE has required a derivation carried out afresh for each model.
We remove that step for two-component systems with conservative polynomial differential nonlinearities on a one-dimensional periodic domain, deriving closed form formulas for both coefficients.
They are expressed directly in terms of the array of nonlinear PDE coefficients and the spectral data of the linearization, its critical eigenvectors, adjoint eigenvectors, and non-critical resolvents.
We give verifiable conditions under which the underlying center manifold reduction is valid.
A structural condition on the principal part of the linearization, together with sign conditions on the trace and determinant of its Fourier symbol, yields the resolvent estimate and the spectral gap that the reduction requires.
A companion implementation evaluates the two normal form coefficients from the coefficient array and verifies the assumptions for a given system.
We specialize the formulas to two applications: a strain-gradient regularization of the nonlinear $p$-system, whose quadratic case contains a previously studied model as a special case, and a first order conservative bilinear family, not previously analyzed, in which every bifurcation scenario allowed by the general classification is realized.
\end{abstract}
\maketitle
\setcounter{tocdepth}{3}
\tableofcontents

\section{Introduction}
Hopf bifurcation in $O(2)$-equivariant PDEs is a standard local mechanism by which spatially periodic systems generate traveling and standing waves.
The group $O(2)$, generated by translations and a reflection, is the natural symmetry group for systems on a one-dimensional periodic domain, and arises in nonlinear wave systems, pattern-forming fluid models, and optical systems \cite{yao20142,monteiro2014transverse,crawford1991symmetry,dangelmayr1987takens,budzinskiy2017rotating,budzinskiy2017normal}.
The equivariant Hopf normal form is classical: near a critical conjugate pair of Fourier modes $e^{\pm i m_c x}$ the reduced dynamics is four dimensional and supports traveling waves, which break the reflection symmetry, and standing waves, which preserve it \cite{crawford1991symmetry,golubitsky2012singularities,haragus2010local,vangils1986hopf}.
Their stability and selection are governed by two complex cubic coefficients, denoted here by $\zeta$ and $\xi$.

General center manifold and normal form theory expresses them abstractly, through multilinear maps, resolvents, and the center manifold corrections obtained from homological equations, which must still be evaluated for each model \cite{marsden1976hopf,haragus2010local,guo2022equivariant,jiang2020formulation}.
In concrete PDEs the passage from these implicit expressions to usable coefficients has been carried out case by case \cite{yao20142,yaoliu2019}.
For reaction-diffusion systems, local and nonlocal, rotating and standing waves on a circle have been analyzed in \cite{farr1992rotating,ogawa2007degenerate,fukushima2008selection,caishisu2026rotating}.
The reduction method we use is the classical one.
What is new is the class: closed forms for conservative two-component systems with differential nonlinearities.

We consider two-component systems on the circle of the form
\[
\frac{d}{dt}\begin{bmatrix}u\\v\end{bmatrix}
=
L_\lambda \begin{bmatrix}u\\v\end{bmatrix}
+
G\begin{bmatrix}u\\v\end{bmatrix},
\]
where $L_\lambda$ is an $O(2)$-equivariant differential operator and $G$ is an $O(2)$-equivariant conservative polynomial differential nonlinearity encoded by a multi-indexed coefficient array $\{\mathbf a_{\alpha_1,\alpha_2}\}$.
For this class we give closed form expressions for $\zeta$ and $\xi$ in terms of the array and the eigenvectors, adjoint eigenvectors, and non-critical resolvents of the Fourier symbols of $L_\lambda$, so that no system specific normal form derivation is required once a PDE is placed in the class.
Beyond the formulas, the explicit reduction reveals two structural features that the implicit theory hides.
The quadratic center-manifold correction is computed by inverting a single matrix, so it needs only that this matrix be invertible, not diagonalizable, and the reduction stays valid even where a standard eigenvector calculation would break down.
Translation by $\pi/m_c$ reverses the sign of every critical mode, so the even-order interactions of the critical modes vanish.
The reflection then relates the coefficients of the two amplitude equations, and together these make the cubic coefficients explicit.
The strategy of reading bifurcation behavior off explicit formulas in the linear spectrum and the nonlinear coefficient array has been effective for related problems \cite{csengul2024first,csengul2025effect}.
Here it is carried out for the $O(2)$-Hopf case.

The reduction itself is not automatic for this operator class, and we make its validity explicit.
The center manifold construction needs a resolvent estimate and a spectral gap that separates the critical Fourier modes from the rest of the spectrum \cite{henry1981geometric,chow1982methods}.
Both follow from a structural condition on the principal part of the linearization together with sign conditions on the trace and determinant of its Fourier symbol.

The PDE class we consider includes hyperbolic conservation laws with strain-gradient regularization and Kuramoto--Sivashinsky-type dynamics \cite{slemrod1983admissibility,lefloch2002hyperbolic,engel2013low,palchoudhury2015self,kuramoto2003chemical,sivashinsky1980flame}.
Because the dynamics is restricted to a fixed-mean periodic phase space, the bifurcation studied here is distinct from instabilities in which a large-scale conserved mode takes part in the dynamics, such as the conserved-Hopf instability \cite{greve2024amplitude}.
As a first application we treat a strain-gradient regularization of a one-dimensional conservation law whose stress--strain law carries both quadratic and cubic terms.
Its quadratic case contains a previously studied model as a special case of a substantially richer family \cite{yao20142,liyao2015}.
The closest comparison is Li and Yao \cite{liyao2015}, who treat the same regularized system for general stress laws and arbitrary critical wavenumbers and obtain explicit coefficients for that scalar family.
Our formulas reproduce theirs on the overlap and are derived instead from the nonlinear coefficient array and the Fourier resolvents, so they apply to the whole two-component class rather than to a stress--strain model.
We then enlarge the nonlinearity within a conservative bilinear family and exhibit system data realizing every bifurcation scenario permitted by the general classification.

The paper is organized as follows.
Section~\ref{sec:main} formulates the PDE class and the spectral hypotheses, and Section~\ref{sec:results} states the center manifold reduction, the $O(2)$-Hopf normal form, and the classification of bifurcating traveling and standing waves.
Section~\ref{sec:formulas} gives the explicit formulas for $\zeta$ and $\xi$, and Section~\ref{sec:app1} applies them to two conservative models, including the bilinear family of Subsection~\ref{sec:app1_TW} that realizes the full classification.
The proofs are given in Section~\ref{sec:proof}, with auxiliary spectral and stability calculations collected in the appendices.
 
\section{Problem Setup and Assumptions}\label{sec:main}
We consider the following nonlinear system of PDEs
\begin{equation}\label{main}
	\frac{d}{dt} \begin{bmatrix} u \\ v \end{bmatrix} = L_\lambda \begin{bmatrix} u \\ v \end{bmatrix} + G\left( \begin{bmatrix} u \\ v \end{bmatrix} \right),
\end{equation}
where $\{L_\lambda\}_{\lambda\in\mathbb{R}}$ is a family of linear operators depending on a real control parameter $\lambda$, and $G$ is a $\lambda$ independent nonlinear operator.
The spatial domain is the circle $\mathbb T = \mathbb R / 2\pi\mathbb Z$, and we work on the mean-zero phase space $\mathcal{H} := \dot H^0_{\mathrm{per}}(0,2\pi)^2$, where $\dot H^s_{\mathrm{per}}(0,2\pi)$ is the zero-mean periodic Sobolev space of order $s\ge0$ and $\mathcal{H}^s := \dot H^s_{\mathrm{per}}(0,2\pi)^2$, that is,
\[
\int_{0}^{2\pi} \begin{bmatrix} u \\ v \end{bmatrix} dx = 0.
\]
We use the $L^2$ inner product on $\mathcal{H}$ and the standard Hermitian inner product on $\mathbb{C}^2$,
\[
\langle (u_1,v_1), (u_2,v_2) \rangle_{L^2} = \int_0^{2\pi} \bigl( u_1\overline{u_2} + v_1\overline{v_2} \bigr)\,dx,
\qquad
\langle \mathbf{q}, \mathbf{p} \rangle_{\mathbb{C}^2} = q_1\overline{p_1} + q_2\overline{p_2}.
\]
The translation and reflection operators
\[
T_h \begin{bmatrix} u(x) \\ v(x) \end{bmatrix} = \begin{bmatrix} u(x+h) \\ v(x+h) \end{bmatrix}, \qquad
R \begin{bmatrix} u(x) \\ v(x) \end{bmatrix} = \begin{bmatrix} u(-x) \\ -v(-x) \end{bmatrix},
\]
generate an $O(2)$ action (i.e.\ $R^2 = \mathrm{id}$ and $T_h R = R T_{-h}$).
\begin{definition}\label{O2equivariance}
The system \eqref{main} is called \emph{$O(2)$-equivariant} if $(L_\lambda + G)$ commutes with both $T_h$ and $R$, for every $\lambda\in\mathbb{R}$ and all $h\in\mathbb{T}$.
\end{definition}

$O(2)$-equivariance forces $L_\lambda = (L_{ij})_{i,j=1,2}$ to have the following block structure
\begin{equation}\label{Lij}
	\begin{aligned}
	& L_{11} = \sum_{k=0}^{m_L} b_{1,2k} \partial_{x}^{2k}, \qquad
	&& L_{12} = \sum_{k=0}^{m_L} b_{1,2k+1} \partial_{x}^{2k+1}, \\
	& L_{21} = \sum_{k=0}^{m_L} b_{2,2k+1} \partial_{x}^{2k+1}, \qquad
	&& L_{22} = \sum_{k=0}^{m_L} b_{2, 2k} \partial_{x}^{2k}
	\end{aligned}
\end{equation}
with $m_L\ge 1$ and $b_{i,k}=b_{i,k}(\lambda)\in\mathbb{R}$ smooth in $\lambda$.

\begin{assumption}\phantomsection\label{cond:D}
\begin{enumerate}[label=(\roman*)]
\item\label{D1}
The principal part of $L$ is diagonal and nondegenerate:
\begin{equation}\label{eq:condD}
b_{1,2m_L}\ne0,\qquad b_{2,2m_L}\ne0,\qquad b_{1,2m_L+1}=b_{2,2m_L+1}=0.
\end{equation}
\item\label{D2}
The bifurcation parameter enters below principal order: $r_\lambda\le
2m_L-1$, where $r_\lambda$ is the highest differential order whose
coefficient in $L_\lambda$ depends on $\lambda$.
\end{enumerate}
\end{assumption}
\noindent

Each component of the nonlinear operator $G=(g_1,g_2)^T$ is of the form
\begin{equation}\label{G}
	g_i(u, v) = \sum_{(\alpha_1,\alpha_2)} a_{\alpha_1, \alpha_2}^i \, (D u)^{\alpha_1} (D v)^{\alpha_2}, \qquad i = 1,2,
\end{equation}
where $\alpha_1,\alpha_2$ belong to the multi-index set
\[
\Lambda = \left\{ \alpha = (\alpha_j)_{j\ge0} \in \mathbb{Z}_{\ge 0}^{\infty} : \abs{\alpha} = \sum_{j\ge0} \alpha_j < \infty \right\},
\]
and $(D u)^{\alpha} = \prod_{j\ge0} (\partial_x^{j} u)^{\alpha_j}$.
For the two multi-indices $\alpha_1,\alpha_2$ in~\eqref{G} we write $\alpha_{k,j}$ for the $j$-th entry of $\alpha_k$, $k=1,2$.

The sum in \eqref{G} contains only pairs with $|\alpha_1|+|\alpha_2|\ge 2$.
We write $\mathbf{a}_{\alpha_1,\alpha_2} = (a^1_{\alpha_1,\alpha_2},\, a^2_{\alpha_1,\alpha_2})^T$ for the coefficient vector.

The derivative loss of $G$,
\begin{equation}\label{rG}
r_G := \max\bigl\{
j : \alpha_{k,j} > 0 \text{ and } a^i_{\alpha_1,\alpha_2} \neq 0 \text{ for some } i,k \in \{1,2\} \bigr\},
\end{equation}
is the highest derivative order appearing in $G$.
The maximum is finite by Assumption~\ref{ass:S}\ref{S1} below, and we set $r_G=0$ when $G\equiv0$.

\begin{definition}\label{defm0m1}
	For $\alpha \in \Lambda \subset \left( \mathbb{Z}_{\ge 0} \right)^{\infty}$, we define the following non-negative integers
	\begin{equation}\label{m0123}
	\mathfrak{m}_e(\alpha) = \sum_{i\ge0} \alpha_{2i}, \quad 
	\mathfrak{m}_o(\alpha) = \sum_{i\ge0} \alpha_{2i+1}, \quad
	\mathfrak{m}_t(\alpha) = \sum_{i\ge 0} i\alpha_i
	\end{equation}
	which count even order, odd order, and total order of derivatives in $(Du)^\alpha$ or $(Dv)^\alpha$, respectively. Moreover,
	\[
	\mathfrak m_e(\alpha) + \mathfrak m_o(\alpha) = \abs{\alpha}.
	\]
\end{definition}
For example, $\alpha=(2,0,1,0,\dots)$ gives $(Du)^\alpha = u^2 u_{xx}$ with $\mathfrak{m}_e(\alpha)=3$, $\mathfrak{m}_o(\alpha)=0$, $\mathfrak{m}_t(\alpha)=2$.

\begin{assumption}\phantomsection\label{ass:S}
\begin{enumerate}[label=(\roman*)]
	\item \label{S1} $G$ is a differential polynomial, i.e., the sums \eqref{G} are over a finite set of multi indices.
	\item\label{S2} 
	\begin{equation}\label{syr_G}
		\begin{aligned}
		& a^1_{\alpha_1, \alpha_2} = 0 \quad \text{if } \mathfrak{m}_o(\alpha_1) + \mathfrak{m}_e(\alpha_2) \equiv 1 \pmod{2}, \\
		& a^2_{\alpha_1, \alpha_2} = 0 \quad \text{if } \mathfrak{m}_o(\alpha_1) + \mathfrak{m}_e(\alpha_2) \equiv 0 \pmod{2}.
	\end{aligned}
\end{equation}
\item \label{S3} $G$ is conservative, i.e., $g_i = \partial_x f_i$, for some polynomial $f_i$, $i=1,2$.
\item\label{S4} The order condition
\begin{equation}\label{rGsL}
r_G \le 2m_L - 1.
\end{equation}
is satisfied.
\end{enumerate}
\end{assumption}
Condition~\ref{S1} is stronger than the cubic coefficients require (see Section~\ref{sec:outlook}).
Only $G$ is assumed conservative.
Every term of $L$ of order at least one is a derivative, so~\eqref{main} is in conservation form exactly when $b_{1,0}=b_{2,0}=0$, as in both applications of Section~\ref{sec:app1}.

\begin{remark}\label{rem:euler}
Condition~\ref{S3} is algorithmically checkable: a differential polynomial $g$ with $g(0,0)=0$ is a total $x$-derivative if and only if the Euler--Lagrange operators
\[
E_{w}(g) = \sum_{r\ge 0}(-\partial_x)^r \frac{\partial g}{\partial (\partial^r_x w)}
\]
vanish for $w=u$ and $w=v$ \cite{olver1993applications}.
In particular, $G$ satisfies~\ref{S3} if and only if $E_u(g_i)=E_v(g_i)=0$ for $i=1,2$.
\end{remark}

\begin{lemma}\label{lem:syr_G}
The operator $G$ is $O(2)$-equivariant if and only if \eqref{syr_G} holds.
\end{lemma}
\noindent The proof is given in Appendix~\ref{app:spectral}.

We regard $L$ as an unbounded operator on $\mathcal{H}$ with domain $D(L) := \mathcal{H}^{2m_L}$.
Since $H^s_{\mathrm{per}}(0,2\pi)$ is a Banach algebra for $s>\tfrac12$, the map $G\colon\mathcal{H}^{2m_L}\to\mathcal{H}$ is well-defined and smooth by~\eqref{rGsL}, see Section~\ref{sec:CM_existence} for details.

\begin{theorem}\label{eigenvalue_theorem}
Suppose Assumption~\ref{cond:D}\ref{D1} holds.
The spectrum of $L$ on $\mathcal{H}$ consists of the eigenvalues 
$$\{\beta_{m,n} : m\in\mathbb{Z}_{\ne0},\, n\in\{1,2\}\},$$
which solve the characteristic equation
\begin{equation}\label{char}
	\beta^2 - \tr M_m\, \beta + \det M_m = 0,
\end{equation}
where
\begin{equation}\label{Mm}
	M_m(\lambda) =
	\sum_{k=0}^{m_L}(-1)^k
	\begin{bmatrix*}[l]
	b_{1,2k}(\lambda)\, m^{2k} & i\, b_{1,2k+1}(\lambda)\, m^{2k+1}\medskip \\
	i\, b_{2,2k+1}(\lambda)\, m^{2k+1} & b_{2,2k}(\lambda)\, m^{2k}
	\end{bmatrix*},
\end{equation}
is the Fourier symbol of $L_\lambda$ ($\partial_x\mapsto im$).
We write $M_m$ when $\lambda$ is fixed.
For any root $\beta_{m,n}$, every nonzero $\mathbf q_{m,n}\in\ker(M_m-\beta_{m,n}I_2)$ gives an eigenfunction $e_{m,n}(x)=e^{imx}\mathbf q_{m,n}$.
When a repeated root is defective, there is only one eigenvector on that block.
Moreover, $\beta_{m,n} = \overline{\beta_{-m,n}}$.
In particular, either
\[
\beta_{m,1} = \beta_{-m,1},\quad \beta_{m,2} = \beta_{-m,2} \in \mathbb{R},
\]
or
\[
\beta_{m,1} = \overline{\beta_{m,2}} = \overline{\beta_{-m,1}} = \beta_{-m,2} \notin \mathbb{R}.
\]
\end{theorem}

\noindent The proof is given in Appendix~\ref{app:spectral}.

\begin{lemma}\label{lem:adjoint_eigen}
For each $m\ne0$ with $\beta_{m,1}\ne\beta_{m,2}$, the eigenvectors on that block may be normalized so that the following statements hold.
The adjoint eigenvectors $e_{m,n}^* = e^{imx}\mathbf{q}_{m,n}^*$ satisfy
\begin{equation}\label{Mm_adjoint_eigen}
M_{-m}^T\,\mathbf{q}_{m,n}^* = \beta_{-m,n}\,\mathbf{q}_{m,n}^*,
\end{equation}
and the orthogonality relations
\begin{equation}\label{orth}
\langle e_{m,n},\,e_{m',n'}^*\rangle_{L^2} = \delta_{mm'}\delta_{nn'}, \qquad
\langle\mathbf{q}_{m,n},\mathbf{q}_{m,n'}^*\rangle_{\mathbb{C}^2} = \frac{\delta_{nn'}}{2\pi},
\end{equation}
hold for all such $m,m'$ and for $n,n'\in\{1,2\}$.
In particular, these relations hold on the critical blocks under Assumption~\ref{hopf_assumption}.
\end{lemma}
The proof is given in Appendix~\ref{app:spectral}.

\begin{assumption}\label{hopf_assumption}
There exist $m_c \in \mathbb{Z}_{\ge 1}$ and $\lambda_c \in \mathbb{R}$ such that:
\begin{enumerate}[label=(\roman*)]
\item $\tr M_{m_c}(\lambda_c) = 0$.
\item \label{ass:det} $\det M_m(\lambda_c) > 0$ for all $m \in \mathbb{Z}_{\ne 0}$.
\item $\dfrac{d}{d\lambda} \tr M_{m_c}(\lambda_c) \ne 0$.
\item $\tr M_m(\lambda_c) < 0$ for all $m \in \mathbb{Z}_{\ne 0}$ with $m \ne \pm m_c$.
\end{enumerate}
\end{assumption}

\begin{remark}\label{rem:meanzero-necessary}
The zero-mean space is invariant because $G$ is conservative and $L_\lambda$ has constant coefficients.
When $M_0=0$, as in Section~\ref{sec:app1}, choosing this space fixes the two conserved means and removes the associated neutral modes, leaving an isolated four-dimensional $O(2)$-Hopf center subspace.
\end{remark}

\begin{lemma}\label{lem:PES}
Define the sets of critical and stable mode indices by
\begin{equation}\label{eq:Mcrit}
\mathfrak{M}_{\mathrm{center}} = \{(m_c,1),(m_c,2),(-m_c,1),(-m_c,2)\},
\qquad
\mathfrak{M}_{\mathrm{stable}} =
\bigl(\mathbb{Z}_{\ne0} \times \{1,2\}\bigr) \setminus \mathfrak{M}_{\mathrm{center}}.
\end{equation}
Under Assumptions~\ref{cond:D}\ref{D1} and~\ref{hopf_assumption}, the eigenvalues of $L(\lambda_c)$ satisfy
\begin{equation}\label{PES}
\begin{aligned}
& \Re\beta_{m,n}(\lambda_c) = 0, && \forall (m,n)\in\mathfrak{M}_{\mathrm{center}},\\
& \Im\beta_{m,n}(\lambda_c) \ne 0, && \forall (m,n)\in\mathfrak{M}_{\mathrm{center}},\\
& \frac{d}{d\lambda}\Re\beta_{m,n}(\lambda_c) \ne 0, && \forall (m,n)\in\mathfrak{M}_{\mathrm{center}},\\
& \sup_{(m,n)\in\mathfrak{M}_{\mathrm{stable}}}\Re\beta_{m,n}(\lambda_c) < 0.
\end{aligned}
\end{equation}
In particular, the center subspace is exactly four-dimensional.
\end{lemma}
The proof is given in Appendix~\ref{app:spectral}.
The first three statements in Lemma~\ref{lem:PES} show that the system undergoes a Hopf bifurcation at $\lambda=\lambda_c$ with critical modes $e_{\pm m_c,1}$ and $e_{\pm m_c,2}$ and frequency $\omega := \Im\beta_{m_c,1}(\lambda_c) = \sqrt{\det M_{m_c}(\lambda_c)} > 0$.
From now on we write $\beta_1 := \beta_{m_c,1}$.
The last statement is a spectral gap condition that ensures the existence of a center manifold and the validity of the reduction.

\begin{remark}\label{rem:mc_unique}
When $m_L=1$, Assumption~\ref{hopf_assumption}~\textit{(i)} and~\textit{(iv)} force $m_c=1$ since $\tr M_m$ is strictly decreasing in $m^2$.
Higher critical wavenumber $m_c\ge2$ is possible only when $m_L\ge2$ i.e.\ when the principal part of $L$ is at least fourth order.
\end{remark}

\begin{lemma}\label{lem:crit_norm}
Suppose Assumptions~\ref{cond:D}\ref{D1} and~\ref{hopf_assumption} hold and let $\lambda=\lambda_c$.
The vector $R\,\overline{\mathbf{q}_{m_c,1}}$ is an eigenvector of $M_{m_c}$ for the eigenvalue $\beta_{m_c,2}=\overline{\beta_{m_c,1}}$, so that
\begin{equation}\label{eq:crit_norm}
R\,\overline{\mathbf{q}_{m_c,1}} = \gamma\,\mathbf{q}_{m_c,2} \quad\text{for some } \gamma\in\mathbb{C}\setminus\{0\}.
\end{equation}
We normalize $\mathbf{q}_{m_c,2}$ so that $\abs{\gamma}=1$, and set $\mathbf{q}_{-m_c,n}:=\overline{\mathbf{q}_{m_c,n}}$, equivalently $e_{-m_c,n}=\overline{e_{m_c,n}}$, for $n=1,2$.
Then
\[
Re_{m_c,1} = \overline{\gamma}\, e_{-m_c,2}, \qquad Re_{m_c,2} = \overline{\gamma}\, e_{-m_c,1}.
\]
\end{lemma}
\begin{proof}
By~\eqref{Mm} the diagonal entries of $M_m$ are real and even in $m$, and the off-diagonal entries are purely imaginary and odd in $m$.
Hence $RM_mR = M_{-m} = \overline{M_m}$, where $R$ also denotes its action $\operatorname{diag}(1,-1)$ on coefficient vectors.
Therefore $M_{m_c}R\,\overline{\mathbf{q}_{m_c,1}} = R\,\overline{M_{m_c}\mathbf{q}_{m_c,1}} = \overline{\beta_{m_c,1}}\,R\,\overline{\mathbf{q}_{m_c,1}}$, and $\overline{\beta_{m_c,1}}=\beta_{m_c,2}$ is simple, since $\beta_{m_c,1}\ne\beta_{m_c,2}$ are purely imaginary and nonzero by Lemma~\ref{lem:PES}.
This gives~\eqref{eq:crit_norm}.
Conjugating~\eqref{eq:crit_norm} gives $R\mathbf{q}_{m_c,1} = \overline{\gamma}\,\overline{\mathbf{q}_{m_c,2}} = \overline{\gamma}\,\mathbf{q}_{-m_c,2}$, and applying $R$ to~\eqref{eq:crit_norm} and using $\gamma^{-1}=\overline{\gamma}$ gives $R\mathbf{q}_{m_c,2} = \overline{\gamma}\,\overline{\mathbf{q}_{m_c,1}} = \overline{\gamma}\,\mathbf{q}_{-m_c,1}$, which are the stated relations since $(Re_{m,n})(x) = e^{-imx}R\mathbf{q}_{m,n}$.
\end{proof}

\section{Main Results}\label{sec:results}
\begin{theorem}\label{thm:cm_reduction}
Suppose Assumptions~\ref{cond:D}, \ref{ass:S} and~\ref{hopf_assumption} hold, and fix an integer $k\ge2$.
Then, for $\lambda$ in a neighborhood of $\lambda_c$ depending on $k$, the system~\eqref{main} possesses a four-dimensional local center manifold $\mathcal{W}^c(\lambda)$ of class $C^k$: it is locally invariant, contains all solutions of~\eqref{main} that remain sufficiently small for all $t\in\mathbb{R}$, and at $\lambda=\lambda_c$ it is tangent at the origin to the center subspace $\operatorname{span}\{e_{\pm m_c,1},\,e_{\pm m_c,2}\}$.
The manifold may be chosen $O(2)$-equivariant and is locally attracting.
\end{theorem}

\begin{theorem}\label{thm:main}
Under the hypotheses of Theorem~\ref{thm:cm_reduction} with $k\ge5$, and with the critical eigenvectors normalized as in Lemma~\ref{lem:crit_norm}, there is a near-identity, $O(2)$-equivariant change of coordinates that brings the flow on the center manifold into the normal form
\begin{equation}\label{reduced_system}
\begin{aligned}
\frac{dZ_1}{dt} &= \beta_{m_c,1}(\lambda)\, Z_1 + Z_1\bigl(\zeta|Z_1|^2 + \xi|Z_2|^2\bigr) + O(\abs{\lambda-\lambda_c}\abs{Z}^3 + \abs{Z}^5), \\
\frac{dZ_2}{dt} &= \beta_{m_c,2}(\lambda)\, Z_2 + Z_2\bigl(\bar{\xi}|Z_1|^2 + \bar{\zeta}|Z_2|^2\bigr) + O(\abs{\lambda-\lambda_c}\abs{Z}^3 + \abs{Z}^5),
\end{aligned}
\end{equation}
where $\abs{Z} := \left(\abs{Z_1}^2 + \abs{Z_2}^2\right)^{1/2}$.
Here $Z_1=A+O(\abs{(A,B)}^3)$ and $Z_2=B+O(\abs{(A,B)}^3)$ are near-identity corrections of the critical amplitudes $A = \langle\psi, e_{m_c,1}^*\rangle_{L^2}$ and $B = \langle\psi, e_{m_c,2}^*\rangle_{L^2}$, and the coefficients $\zeta$ and $\xi$ are evaluated at $\lambda=\lambda_c$.
The coefficients $\zeta$ and $\xi$ in~\eqref{reduced_system} decompose as
\[
\zeta = \hat{s}_{1,1} + \hat{c}_{1,1}, \qquad \xi = \hat{s}_{1,2} + \hat{c}_{1,2},
\]
with universally valid matrix-resolvent formulas given in Section~\ref{sec:formulas}.
Equivalent scalar eigenbasis formulas are also given there when $M_{2m_c}$ has distinct eigenvalues.
\end{theorem}
The proofs of Theorem~\ref{thm:cm_reduction} and Theorem~\ref{thm:main} are given in Section~\ref{sec:proof}.

\begin{remark}\label{rem:norm_invariance}
The values of $\zeta$ and $\xi$ depend on the scaling of the critical eigenvectors: replacing $\mathbf{q}_{m_c,1}$ by $s\,\mathbf{q}_{m_c,1}$ with $s\ne0$, and rescaling $\mathbf{q}_{m_c,2}$ accordingly via Lemma~\ref{lem:crit_norm}, rescales $(\zeta,\xi)$ to $(\abs{s}^2\zeta,\abs{s}^2\xi)$.
The signs of $\zetaR$ and $\xiR$, and hence the classification below, are independent of this choice.
\end{remark}
 
The possible bifurcation scenarios are determined by the signs of $\zetaR$ and $\xiR$, where
\[
\zetaR:=\Re\zeta \quad\text{and}\quad \xiR:=\Re\xi.
\]
\begin{definition}\label{def:tw_sw}
The cubic truncation of~\eqref{reduced_system} admits two primary bifurcating solutions:
\begin{enumerate}[label=(\roman*)]
\item A \textbf{traveling wave} (TW): $(Z_1,Z_2) = (r\,e^{i\theta t},\, 0)$ with $r^2 = -\Re\beta_{m_c,1}/\zetaR > 0$ and $\theta = \Im\beta_{m_c,1} + O(r^2)$.
\item A \textbf{standing wave} (SW): $(Z_1,Z_2) = r(e^{i\theta t},\, e^{-i\theta t})$ with $r^2 = -\Re\beta_{m_c,1}/(\zetaR+\xiR) > 0$ and $\theta = \Im\beta_{m_c,1} + O(r^2)$.
\end{enumerate}
\end{definition}
In the original PDE variables, to the leading order in $r$, the TW and SW take the form
\[
\begin{bmatrix} u & v \end{bmatrix}^T = 2r\,\Re\bigl(e^{i(m_c x + \theta t)}\mathbf{q}_{m_c,1}\bigr)
\qquad\text{(TW)},
\]
\[
\begin{bmatrix} u & v \end{bmatrix}^T = 2r\,\Re\bigl(e^{i(m_c x + \theta t)}\mathbf{q}_{m_c,1} + e^{i(m_c x - \theta t)}\mathbf{q}_{m_c,2}\bigr)
\qquad\text{(SW)}.
\]
Figure~\ref{fig:app1_sim} shows both in a direct simulation.

The dynamic behavior of the reduced system~\eqref{reduced_system}
is summarized in the following proposition.
\begin{proposition}[{\cite{crawford1991symmetry,vangils1986hopf}}]\label{prop:classification}
The bifurcation structure
of~\eqref{reduced_system} is classified by the signs of $\zetaR$, $\xiR$, $\zetaR+\xiR$, 
$\zetaR-\xiR$, and the sign of $\Re\beta_{m_c,1}(\lambda)$ as given in Figure~\ref{fig:bif}, Figure~\ref{fig:phase} and Table~\ref{table:signs}.
\end{proposition}
A proof is given in \cite[Section~4.1]{crawford1991symmetry}.
In Table~\ref{table:signs}, existence and stability refer to the cubic truncation of~\eqref{reduced_system}.
Stability is understood orbitally, modulo temporal phase and the $O(2)$ action.
When $\zetaR\ne0$ and $\zetaR\pm\xiR\ne0$, the corresponding equilibria of the truncated amplitude equations are hyperbolic in the radial variables, so both branches and their stability types persist under the higher-order remainder \cite{vangils1986hopf,crawford1991symmetry}.
Stability on the center manifold then implies orbital stability for~\eqref{main}, by the local attractivity of the center manifold afforded by the spectral gap \cite[Theorem~3.23]{haragus2010local}, \cite{ptd}.

\begin{figure}[ht]
\centering
\begin{tikzpicture}[
  scale=0.60,
  thick,
  >=Stealth,
  font=\small,
  stable/.style   = {solid,  line width=1.6pt},
  unstable/.style = {dashed, line width=1.2pt, dash pattern=on 5pt off 3pt},
  saddle/.style   = {dotted, line width=1.2pt},
]


\draw[gray!60, dotted, ->] (-5.6,0) -- (5.6,0) node[black, right] {$\zetaR$};
\draw[->] (0,-5.6) -- (0,5.6) node[black, above] {$\xiR$};

\draw[gray!80, line width=0.8pt] (-5, 5) -- (5,-5)
      node[black, right, yshift= 2pt, font=\footnotesize] {$\zetaR+\xiR=0$};
\draw[gray!80, line width=0.8pt] (-5,-5) -- (5, 5)
      node[black, right, font=\footnotesize] {$\zetaR-\xiR=0$};

\foreach \lbl/\x/\y in {I/1.5/0, II/1/2, III/-1/2,
                          IV/-1.5/0, V/-1/-2, VI/1/-2}{
  \node at (\x,\y) {\lbl};
}


\begin{scope}[shift={(4.25, -0.25)}]
  \draw[gray!70] (-0.75,0) -- (0.75,0);
  \draw[unstable,   swred]  (0,0) arc[start angle=5,   end angle=90, radius=1.1]
        node[right, yshift=6pt, text=swred, font=\scriptsize] {SW};
  \draw[saddle, twblue] (0,0) arc[start angle=35,  end angle=80, radius=1.7]
        node[left,  xshift=2pt, text=twblue, font=\scriptsize] {TW};
\end{scope}

\begin{scope}[shift={(2.5, 3.7)}]
  \draw[gray!70] (-0.75,0) -- (0.75,0);
  \draw[saddle,   swred]  (0,0) arc[start angle=5,   end angle=90, radius=1.1]
        node[right, yshift=6pt, text=swred, font=\scriptsize] {SW};
  \draw[unstable, twblue] (0,0) arc[start angle=35,  end angle=80, radius=1.7]
        node[left,  xshift=2pt, text=twblue, font=\scriptsize] {TW};
\end{scope}

\begin{scope}[shift={(-2.0, 3.7)}]
  \draw[gray!70] (-0.75,0) -- (0.75,0);
  \draw[saddle, swred]  (0,0) arc[start angle=5,   end angle=90, radius=1.1]
        node[right, xshift=-8pt, yshift=6pt, text=swred, font=\scriptsize] {SW};
  \draw[saddle, twblue] (0,0) arc[start angle=175, end angle=90, radius=1.1]
        node[left, xshift=8pt,  yshift=6pt, text=twblue, font=\scriptsize] {TW};
\end{scope}

\begin{scope}[shift={(-4.25, -0.25)}]
  \draw[gray!70] (-0.75,0) -- (0.75,0);
  \draw[stable,  swred]  (0,0) arc[start angle=175, end angle=90, radius=1.1]
        node[above, text=swred, font=\scriptsize] {SW};
  \draw[saddle,  twblue] (0,0) arc[start angle=145, end angle=100, radius=1.7]
        node[right, xshift=-2pt, text=twblue, font=\scriptsize] {TW};
\end{scope}

\begin{scope}[shift={(-2.5, -4.5)}]
  \draw[gray!70] (-0.75,0) -- (0.75,0);
  \draw[saddle, swred]  (0,0) arc[start angle=175, end angle=90, radius=1.1]
        node[above, text=swred, font=\scriptsize] {SW};
  \draw[stable, twblue] (0,0) arc[start angle=145, end angle=100, radius=1.7]
        node[right, xshift=-2pt, text=twblue, font=\scriptsize] {TW};
\end{scope}

\begin{scope}[shift={(1.5, -4.5)}]
  \draw[gray!70] (-0.75,0) -- (0.75,0);
  \draw[saddle, twblue] (0,0) arc[start angle=5,   end angle=90, radius=1.1]
        node[right, xshift=-8pt, yshift=6pt, text=twblue, font=\scriptsize] {TW};
  \draw[saddle, swred]  (0,0) arc[start angle=175, end angle=90, radius=1.1]
        node[left, xshift=8pt,  yshift=6pt, text=swred,  font=\scriptsize] {SW};
\end{scope}

\begin{scope}[shift={(7.8,0.5)}]

  \draw[stable]   (0,1.8) -- (1.6,1.8);

  \node[right] at (2.2,1.8) {stable};

  \draw[unstable] (0,0.8) -- (1.6,0.8);

  \node[right] at (2.2,0.8) {unstable};

  \draw[saddle]   (0,-0.2) -- (1.6,-0.2);

  \node[right] at (2.2,-0.2) {saddle};

\end{scope}
\end{tikzpicture}
\caption{Classification of the $O(2)$-Hopf bifurcation scenarios in the $(\zetaR,\xiR)$-plane.
Branches to the right (resp.\ left) of the origin correspond to $\Re\beta_{m_c,1} > 0$ (resp.\ $< 0$).
\label{fig:bif}}
\end{figure}
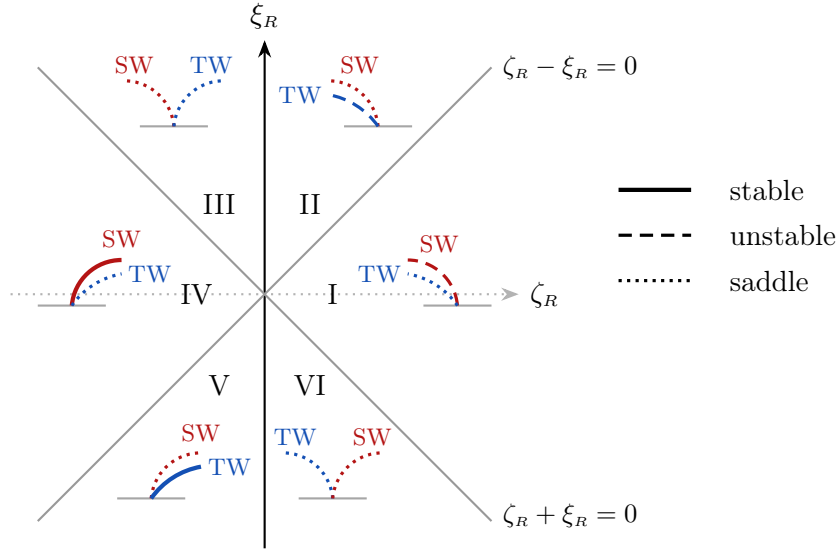

\begin{table}[ht]
\centering
\renewcommand{\arraystretch}{0.75}
\setlength{\tabcolsep}{8pt}
\begin{tabular}{
  >{\centering\arraybackslash}p{1.5cm}   
  >{\raggedright\arraybackslash}p{5.0cm} 
  c                                       
  c                                       
  c                                       
}
\toprule
\textbf{Region} & \textbf{Conditions} & $\Re\beta_{m_c,1}$ & {\textbf{TW}} & {\textbf{SW}} \\
\midrule
\textbf{I}   & $\zetaR+\xiR > 0,\ \zetaR-\xiR > 0$ & $-$ & {saddle}   & {unstable} \\
\addlinespace
\textbf{II}  & $\zetaR > 0,\ \zetaR-\xiR < 0$   & $-$ & {unstable} & {saddle} \\
\addlinespace
\multirow{2}{*}{\textbf{III}} & \multirow{2}{4.0cm}{$\zetaR < 0,\ \zetaR+\xiR > 0$}
  & $-$ & dne              & {saddle} \\
  & & $+$ & {saddle} & dne              \\
\addlinespace
\textbf{IV}  & $\zetaR+\xiR < 0,\ \zetaR-\xiR < 0$ & $+$ & {saddle}   & {stable} \\
\addlinespace
\textbf{V}   & $\zetaR < 0,\ \zetaR-\xiR > 0$   & $+$ & {stable}   & {saddle} \\
\addlinespace
\multirow{2}{*}{\textbf{VI}} & \multirow{2}{4.0cm}{$\zetaR > 0,\ \zetaR+\xiR < 0$}
  & $+$ & dne              & {saddle} \\
  & & $-$ & {saddle} & dne              \\
\bottomrule
\end{tabular}
\caption{Stability of TW and SW solutions depending on the sign of $\Re\beta_{m_c,1}$. dne = does not exist.}
\label{table:signs}
\end{table}

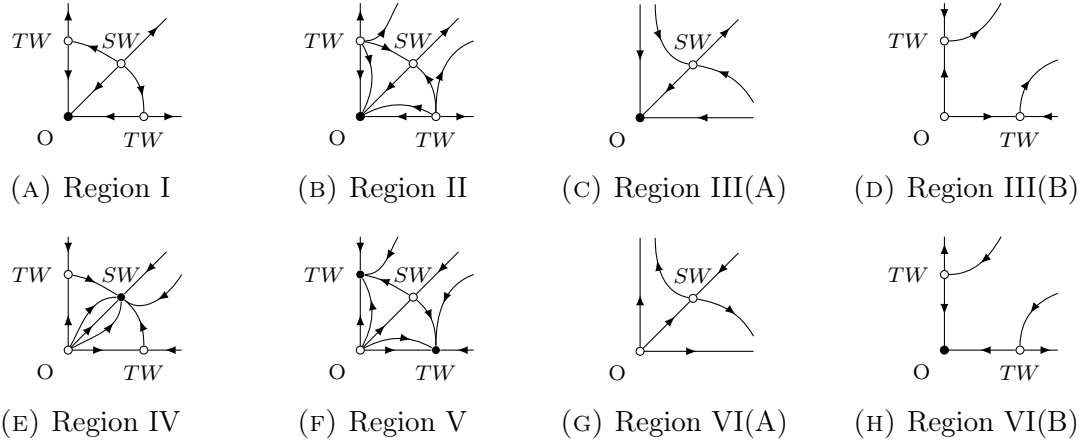
\begin{figure}[tb]
	\centering
	\tikzset{->-/.style={decoration={
				markings,
				mark=at position #1 with {\arrow[scale=1]{latex}}},postaction={decorate}}}
	\tikzset{-<-/.style={decoration={
				markings,
				mark=at position #1 with {\arrowreversed[scale=1]{latex}}},postaction={decorate}}}
	\subcaptionbox{Region I}[.2\textwidth]{
		\begin{tikzpicture}[scale=1]
		\node[circle,fill=white, draw, inner sep=0pt,minimum size=3pt,label=below:{\tiny $TW$}] (E1) at (1,0) {};
		\node[circle,fill=white, draw, inner sep=0pt,minimum size=3pt,label=left:{\tiny $TW$}] (E2) at (0,1) {};
		\node[circle,fill=white,draw, inner sep=0pt,minimum size=3pt,label=above:{\tiny $SW$}] (E3) at (.7,.7) {};
		\draw[-<-=.5] (0,0) to (E2);
		\draw[-<-=.2] (0,1.5) to (E2);
		\draw[-<-=.5] (0,0) to (E1);
		\draw[-<-=.2] (1.5,0) to (E1);
		\draw[-<-=.5] (0,0) to (E3);
		\draw[-<-=.2] (1.3,1.3) to (E3);
		\draw[->-=.7] (E3) to [out=150,in=-10] (E2);
		\draw[->-=.7] (E3) to [out=320,in=90] (E1);
		\node[circle,fill=black, draw, inner sep=0pt,minimum size=3pt,label=below left:{\tiny O}] (E0) at (0,0) {};
		\end{tikzpicture}
	}
	\hspace{2mm}
	\subcaptionbox{Region II}[.2\textwidth]{
		\begin{tikzpicture}[scale=1]
		\node[circle,fill=white, draw,inner sep=0pt,minimum size=3pt,label=below:{\tiny $TW$}] (E1) at (1,0) {};
		\node[circle,fill=white, draw, inner sep=0pt,minimum size=3pt,label=left:{\tiny $TW$}] (E2) at (0,1) {};
		\node[circle,fill=white,draw, inner sep=0pt,minimum size=3pt,label=above:{\tiny $SW$}] (E3) at (.7,.7) {};
		\draw[-<-=.5] (0,0) to (E2);
		\draw[-<-=.2] (0,1.5) to (E2);
		\draw[-<-=.5] (0,0) to (E1);
		\draw[-<-=.2] (1.5,0) to (E1);
		\draw[-<-=.5] (0,0) to (E3);
		\draw[-<-=.2] (1.3,1.3) to (E3);
		\draw[-<-=.4] (E3) to [out=150,in=-10] (E2);
		\draw[-<-=.2] (E3) to [out=320,in=90] (E1);
		\draw[-<-=.7] (0,0) to [out=30,in=150] (E1);
		\draw[-<-=.7] (0,0) to [out=60,in=-60] (E2);
		\draw[-<-=.5] (1.5,1) to [out=200,in=90] (E1);
		\draw[-<-=.5] (.5,1.5) to [out=240,in=0] (E2);
		\node[circle,fill=black, draw, inner sep=0pt,minimum size=3pt,label=below left:{\tiny O}] (E0) at (0,0) {};
		\end{tikzpicture}
	}
	\hspace{2mm}
	\subcaptionbox{Region III(A)}[.2\textwidth]{
		\begin{tikzpicture}[scale=1]
		\node[circle,fill=white,draw, inner sep=0pt,minimum size=3pt,label=above:{\tiny $SW$}] (E3) at (.7,.7) {};
		\draw[-<-=.5] (0,0) to (0, 1.5);
		\draw[->-=.5] (1.5,0) to (0,0);
		\draw[-<-=.5] (0,0) to (E3);
		\draw[-<-=.2] (1.3,1.3) to (E3);
		\draw[->-=.5] (.2, 1.5) to [out=270,in=170] (E3);
		\draw[->-=.7] (1.5,.2) to [out=120,in=350] (E3);
		\node[circle,fill=black, draw, inner sep=0pt,minimum size=3pt,label=below left:{\tiny O}] (E0) at (0,0) {};
		\end{tikzpicture}
	}
	\hspace{2mm}
	\subcaptionbox{Region III(B)}[.2\textwidth]{
		\begin{tikzpicture}[scale=1]
		\node[circle,fill=white, draw,inner sep=0pt,minimum size=3pt,label=below:{\tiny $TW$}] (E1) at (1,0) {};
		\node[circle,fill=white, draw, inner sep=0pt,minimum size=3pt,label=left:{\tiny $TW$}] (E2) at (0,1) {};
		\draw[->-=.7] (0,0) to (E2);
		\draw[->-=.5] (0,1.5) to (E2);
		\draw[->-=.7] (0,0) to (E1);
		\draw[->-=.5] (1.5,0) to (E1);
		\draw[-<-=.5] (1.5,0.75) to [out=200,in=90] (E1);
		\draw[-<-=.5] (0.75,1.5) to [out=240,in=0] (E2);
		\node[circle,fill=white, draw, inner sep=0pt,minimum size=3pt,label=below left:{\tiny O}] (E0) at (0,0) {};
		\end{tikzpicture}
	}
	\par\bigskip
	\subcaptionbox{Region IV}[.2\textwidth]{
		\begin{tikzpicture}[scale=1]
		\node[circle,fill=white, draw, inner sep=0pt,minimum size=3pt,label=below:{\tiny $TW$}] (E1) at (1,0) {};
		\node[circle,fill=white, draw, inner sep=0pt,minimum size=3pt,label=left:{\tiny $TW$}] (E2) at (0,1) {};
		\node[circle,fill=black, inner sep=0pt,minimum size=3pt,label=above:{\tiny $SW$}] (E3) at (.7,.7) {};
		\draw[->-=.5] (0,0) to (E2);
		\draw[->-=.5] (0,1.5) to (E2);
		\draw[->-=.5] (0,0) to (E1);
		\draw[->-=.5] (1.5,0) to (E1);
		\draw[->-=.5] (0,0) to (E3);
		\draw[->-=.5] (1.3,1.3) to (E3);
		\draw[-<-=.5] (E3) to [out=150,in=-10] (E2);
		\draw[-<-=.5] (E3) to [out=320,in=90] (E1);
		\draw[->-=.7] (0,0) to [out=30,in=270] (E3);
		\draw[->-=.7] (0,0) to [out=60,in=180] (E3);
		\draw[->-=.5] (1.5,1) to [out=240,in=320] (E3);
		\node[circle,fill=white, draw, inner sep=0pt,minimum size=3pt,label=below left:{\tiny O}] (E0) at (0,0) {};
		\end{tikzpicture}
	}
	\hspace{2mm}
	\subcaptionbox{Region V}[.2\textwidth]{
		\begin{tikzpicture}[scale=1]
		\node[circle,fill=black,inner sep=0pt,minimum size=3pt,label=below:{\tiny $TW$}] (E1) at (1,0) {};
		\node[circle,fill=black,inner sep=0pt,minimum size=3pt,label=left:{\tiny $TW$}] (E2) at (0,1) {};
		\node[circle,fill=white,draw, inner sep=0pt,minimum size=3pt,label=above:{\tiny $SW$}] (E3) at (.7,.7) {};
		\draw[->-=.5] (0,0) to (E2);
		\draw[->-=.5] (0,1.5) to (E2);
		\draw[->-=.5] (0,0) to (E1);
		\draw[->-=.5] (1.5,0) to (E1);
		\draw[->-=.5] (0,0) to (E3);
		\draw[->-=.5] (1.3,1.3) to (E3);
		\draw[->-=.7] (E3) to [out=150,in=-10] (E2);
		\draw[->-=.5] (E3) to [out=320,in=90] (E1);
		\draw[->-=.7] (0,0) to [out=30,in=150] (E1);
		\draw[->-=.7] (0,0) to [out=60,in=-60] (E2);
		\draw[->-=.5] (1.5,1) to [out=200,in=90] (E1);
		\draw[->-=.5] (.5,1.5) to [out=240,in=0] (E2);
		\node[circle,fill=white, draw, inner sep=0pt,minimum size=3pt,label=below left:{\tiny O}] (E0) at (0,0) {};
		\end{tikzpicture}
	}
	\hspace{2mm}
	\subcaptionbox{Region VI(A)}[.2\textwidth]{
		\begin{tikzpicture}[scale=1]
		\node[circle,fill=white,draw, inner sep=0pt,minimum size=3pt,label=above:{\tiny $SW$}] (E3) at (.7,.7) {};
		\draw[->-=.5] (0,0) to (0, 1.5);
		\draw[-<-=.5] (1.5,0) to (0,0);
		\draw[->-=.7] (0,0) to (E3);
		\draw[->-=.5] (1.3,1.3) to (E3);
		\draw[-<-=.3] (.2, 1.5) to [out=270,in=170] (E3);
		\draw[-<-=.4] (1.5,.2) to [out=120,in=350] (E3);
		\node[circle,fill=white, draw, inner sep=0pt,minimum size=3pt,label=below left:{\tiny O}] (E0) at (0,0) {};
		\end{tikzpicture}
	}
	\hspace{2mm}
	\subcaptionbox{Region VI(B)}[.2\textwidth]{
		\begin{tikzpicture}[scale=1]
		\node[circle,fill=white, draw,inner sep=0pt,minimum size=3pt,label=below:{\tiny $TW$}] (E1) at (1,0) {};
		\node[circle,fill=white, draw, inner sep=0pt,minimum size=3pt,label=left:{\tiny $TW$}] (E2) at (0,1) {};
		\draw[-<-=.5] (0,0) to (E2);
		\draw[-<-=.2] (0,1.5) to (E2);
		\draw[-<-=.5] (0,0) to (E1);
		\draw[-<-=.2] (1.5,0) to (E1);
		\draw[->-=.5] (1.5,0.75) to [out=200,in=90] (E1);
		\draw[->-=.5] (0.75,1.5) to [out=240,in=0] (E2);
		\node[circle,fill=black, draw, inner sep=0pt,minimum size=3pt,label=below left:{\tiny O}] (E0) at (0,0) {};
		\end{tikzpicture}
	}
\caption{Phase portraits in the $(|Z_1|,|Z_2|)$ quadrant for each region of Table~\ref{table:signs}.
Filled nodes are stable.
Open nodes are unstable or saddles.
Sub-cases (A), (B) in Regions III and VI correspond to the sign of $\Re\beta_{m_c,1}$.
\label{fig:phase}}
\end{figure}

\begin{remark}\label{rmk:transition_type}
Assume $\zetaR\neq0$ and $\zetaR \pm \xiR \neq 0$.
In the sense of \cite{ptd}, the transition at $\lambda = \lambda_c$ is of Type-I (continuous) in Regions~IV and~V, and of Type-II (jump) otherwise.
The degenerate boundary cases require higher-order normal form terms, quintic and at $\zetaR=\xiR$ seventh-order, and their classification is surveyed in \cite[Section~4.1]{crawford1991symmetry}.
These statements are local.
In the Type-II case we do not determine where solutions go after they leave a neighborhood of the origin.
In the Type-I case, the attractor-bifurcation theorem applies to the four-real-dimensional center dynamics and yields a local bifurcated attractor with the homology of $S^3$ \cite[Theorem~6.1]{mawang2005bifurcation}, \cite[Theorem~2.2.11]{ptd}.
Indeed, for $\zetaR < 0$ and $\zetaR + \xiR < 0$, at $\lambda = \lambda_c$ we have, up to the $O(\abs{Z}^6)$ contribution of the remainder in~\eqref{reduced_system},
\[
\frac{d}{dt}\abs{Z}^2
= 2\,\zetaR\,\bigl(|Z_1|^4 + |Z_2|^4\bigr) + 4\,\xiR\,|Z_1|^2|Z_2|^2 \le c \abs{Z}^4,
\]
for some $c<0$ and all sufficiently small $(Z_1,Z_2) \neq 0$. 
\end{remark}

\section{Explicit Coefficient Formulas}\label{sec:formulas}
The self-interaction coefficients are given by
\begin{equation}\label{s_hat_k}
	\begin{aligned}
	\hat{s}_{1,1} &= s_{1,1} u_{m_c,1} \abs{u_{m_c,1}}^2 + s_{1,2} u_{m_c,1} \abs{v_{m_c,1}}^2 + s_{1,3} v_{m_c,1} \abs{u_{m_c,1}}^2 \\
	&\quad + s_{1,4} v_{m_c,1} \abs{v_{m_c,1}}^2 + s_{1,5} \overline{u_{m_c,1}}v_{m_c,1}^2 + s_{1,6} u_{m_c,1}^2 \overline{v_{m_c,1}}, \\[4pt]
	\hat{s}_{1,2} &= 2 s_{1,1} u_{m_c,1} \abs{u_{m_c,2}}^2 + s_{1,2} \overline{v_{m_c,2}} ( u_{m_c,1} v_{m_c,2} + u_{m_c,2} v_{m_c,1}) \\
	&\quad + s_{1,3} \overline{u_{m_c,2}} ( v_{m_c,1} u_{m_c,2} + v_{m_c,2} u_{m_c,1}) + 2 s_{1,4} v_{m_c,1} \abs{v_{m_c,2}}^2 \\
	&\quad + 2 s_{1,5} \overline{u_{m_c,2}} v_{m_c,1} v_{m_c,2} + 2 s_{1,6} u_{m_c,1} u_{m_c,2} \overline{v_{m_c,2}}.
	\end{aligned}
\end{equation}
The scalar coefficients $s_{n,k}$ ($n=1,2$, $k=1,\ldots,6$) are defined as follows. 
\begin{equation}\label{s_k_coeff}
	\begin{alignedat}{3}
		s_{n,1} &= \sum_{\substack{|\alpha_1|=3 \\ |\alpha_2|=0}} \sigma_{m_c,n}(\alpha_1,\alpha_2)\,(\mathfrak{m}_e(\alpha_1) - \mathfrak{m}_o(\alpha_1)), \quad &
		s_{n,2} &= \sum_{\substack{|\alpha_1|=1 \\ |\alpha_2|=2}} \sigma_{m_c,n}(\alpha_1,\alpha_2)\,(\mathfrak{m}_e(\alpha_2) - \mathfrak{m}_o(\alpha_2)), \\[4pt]
		s_{n,3} &= \sum_{\substack{|\alpha_1|=2 \\ |\alpha_2|=1}} \sigma_{m_c,n}(\alpha_1,\alpha_2)\,(\mathfrak{m}_e(\alpha_1) - \mathfrak{m}_o(\alpha_1)), \quad &
		s_{n,4} &= \sum_{\substack{|\alpha_1|=0 \\ |\alpha_2|=3}} \sigma_{m_c,n}(\alpha_1,\alpha_2)\,(\mathfrak{m}_e(\alpha_2) - \mathfrak{m}_o(\alpha_2)), \\[4pt]
		s_{n,5} &= \sum_{\substack{|\alpha_1|=1 \\ |\alpha_2|=2}} \sigma_{m_c,n}(\alpha_1,\alpha_2)\,(\mathfrak{m}_e(\alpha_1) - \mathfrak{m}_o(\alpha_1)), \quad &
		s_{n,6} &= \sum_{\substack{|\alpha_1|=2 \\ |\alpha_2|=1}} \sigma_{m_c,n}(\alpha_1,\alpha_2)\,(\mathfrak{m}_e(\alpha_2) - \mathfrak{m}_o(\alpha_2)).
	\end{alignedat}
\end{equation}
Here
\begin{equation}\label{sigma_n}
\sigma_{m,n}(\alpha_1,\alpha_2) := 2\pi\,
\langle \mathbf{a}_{\alpha_1,\alpha_2}, \mathbf{q}_{m,n}^* \rangle_{\mathbb{C}^2}\,
(im_c)^{\mathfrak{m}_t(\alpha_1+\alpha_2)}.
\end{equation}
The first index $m$ in $\sigma_{m,n}$ selects the adjoint eigenvector used in the projection.
The derivative factor remains $im_c$ because every factor being multiplied is a critical mode.

For the quadratic terms, define the polarization
\begin{equation}\label{Q_AB}
\mathcal{Q}_{\alpha_1,\alpha_2}
= \abs{\alpha_1}\, u_{m_c,1}\, u_{m_c,2}^{\abs{\alpha_1}-1}\, v_{m_c,2}^{\abs{\alpha_2}}
+ \abs{\alpha_2}\, u_{m_c,2}^{\abs{\alpha_1}}\, v_{m_c,1}\, v_{m_c,2}^{\abs{\alpha_2}-1},
\end{equation}
with the convention $0\cdot z^{-1}=0$ (so only the summand with $|\alpha_k|\ge1$ contributes).
The coefficients $C_{U,i,n}$ and $C_{V,i,n}$ are the following sums over quadratic multi-indices:
\begin{equation}\label{CUV}
\begin{aligned}
C_{U,1,n} &= \sum_{\abs{\alpha_1}=2,\,\abs{\alpha_2}=0} A_{\alpha_1,\alpha_2,n,1}, \qquad
C_{U,2,n} = \sum_{\abs{\alpha_1}=1,\,\abs{\alpha_2}=1} A_{\alpha_1,\alpha_2,n,1}, \\
C_{V,1,n} &= \sum_{\abs{\alpha_1}=1,\,\abs{\alpha_2}=1} A_{\alpha_1,\alpha_2,n,2}, \qquad
C_{V,2,n} = \sum_{\abs{\alpha_1}=0,\,\abs{\alpha_2}=2} A_{\alpha_1,\alpha_2,n,2}.
\end{aligned}
\end{equation}

\begin{equation}\label{ABalpha}
A_{\alpha_1,\alpha_2,n,j} := (-1)^{\mathfrak{m}_o(\alpha_1+ \alpha_2)}\,
\sigma_{m_c, n}(\alpha_1,\alpha_2)\,\sum_{p\ge 0}(-2)^{p}\,\alpha_{j,p}, \qquad j=1,2.
\end{equation}

We first state the cross-interaction coefficients in a form valid whether or not the second-harmonic matrix $M_{2m_c}$ is diagonalizable.
Set
\begin{equation}\label{eq:g_vectors}
\begin{aligned}
\widehat{\mathbf g}_{AA}
&:=\sum_{|\alpha_1|+|\alpha_2|=2}\mathbf a_{\alpha_1,\alpha_2}\,
u_{m_c,1}^{|\alpha_1|}v_{m_c,1}^{|\alpha_2|}
\,(im_c)^{\mathfrak m_t(\alpha_1+\alpha_2)},\\
\widehat{\mathbf g}_{AB}
&:=\sum_{|\alpha_1|+|\alpha_2|=2}\mathbf a_{\alpha_1,\alpha_2}\,
\mathcal Q_{\alpha_1,\alpha_2}
\,(im_c)^{\mathfrak m_t(\alpha_1+\alpha_2)},
\end{aligned}
\end{equation}
and define the second-harmonic correction vectors
\begin{equation}\label{eq:Phi_vectors}
\mathbf h_{AA}:=(2\beta_1I_2-M_{2m_c})^{-1}\widehat{\mathbf g}_{AA},
\qquad
\mathbf h_{AB}:=-M_{2m_c}^{-1}\widehat{\mathbf g}_{AB}.
\end{equation}
Both inverses exist by~\eqref{PES}.  For $j=1,2$, introduce the row vector
\begin{equation}\label{eq:ell_vectors}
\boldsymbol\ell_{n,j}:=
\begin{pmatrix}
C_{U,1,n}\overline{u_{m_c,j}}+C_{U,2,n}\overline{v_{m_c,j}},&
C_{V,1,n}\overline{u_{m_c,j}}+C_{V,2,n}\overline{v_{m_c,j}}
\end{pmatrix}.
\end{equation}
Then the universally valid matrix formulas are
\begin{equation}\label{eq:c_hat_matrix}
\widehat c_{n,1}=\boldsymbol\ell_{n,1}\mathbf h_{AA},
\qquad
\widehat c_{n,2}=\boldsymbol\ell_{n,2}\mathbf h_{AB}.
\end{equation}

For comparison with traditional modal calculations, suppose now that $M_{2m_c}$ has two distinct eigenvalues and choose its right and adjoint eigenvectors as in Lemma~\ref{lem:adjoint_eigen}.  Expanding the two resolvents in this eigenbasis turns~\eqref{eq:c_hat_matrix} into
\begin{equation}\label{c_hat_k}
\begin{aligned}
\hat{c}_{n,1} &= \sum_{k=1}^2 \Phi_{AA,2m_c,k} \left[ \left( C_{U,1,n}\,u_{2m_c,k} + C_{V,1,n}\, v_{2m_c,k} \right)\overline{u_{m_c,1}} + \left( C_{U,2,n}\, u_{2m_c,k} + C_{V,2,n}\, v_{2m_c,k} \right)\overline{v_{m_c,1}} \right], \\
\hat{c}_{n,2} &= \sum_{k=1}^2 \Phi_{AB,2m_c,k} \left[ \left( C_{U,1,n}\,u_{2m_c,k} + C_{V,1,n}\, v_{2m_c,k} \right)\overline{u_{m_c,2}} + \left( C_{U,2,n}\, u_{2m_c,k} + C_{V,2,n}\, v_{2m_c,k} \right)\overline{v_{m_c,2}} \right],
\end{aligned}
\end{equation}
where
\begin{equation}\label{Phi_coeff}
\begin{aligned}
\Phi_{AA,2m_c,k} &= \frac{1}{2\beta_{m_c,1} - \beta_{2m_c,k}}
\sum_{\abs{\alpha_1} + \abs{\alpha_2} = 2} \sigma_{2m_c,k}(\alpha_1,\alpha_2)\,
u_{m_c,1}^{\abs{\alpha_1}}\, v_{m_c,1}^{\abs{\alpha_2}}, \\
\Phi_{AB,2m_c,k} &= \frac{1}{-\beta_{2m_c,k}}
\sum_{\abs{\alpha_1} + \abs{\alpha_2} = 2} \sigma_{2m_c,k}(\alpha_1,\alpha_2)\,
\mathcal{Q}_{\alpha_1,\alpha_2}.
\end{aligned}
\end{equation}
The denominators $2\beta_{m_c,1}-\beta_{2m_c,k}$ and $\beta_{2m_c,k}$ are nonzero by~\eqref{PES}.
At a defective second-harmonic symbol the scalar summands in~\eqref{c_hat_k}--\eqref{Phi_coeff} are not defined separately, whereas the matrix formula~\eqref{eq:c_hat_matrix} remains regular.

\begin{remark}\label{rem:scaling}
The physical interpretation of the self and cross-interaction coefficients are as follows: $\hat{s}_{1,j}$ (self-interaction) is linear in the cubic coefficients of $G$, and $\hat{c}_{1,j}$ (cross-interaction) is bilinear in the quadratic coefficients.

Thus under a uniform scaling $G\mapsto\epsilon G$, the linearity and bilinearity above give
\begin{equation}\label{eq:zeta_xi_scaling}
\zeta(\epsilon G) = \epsilon\,\hat{s}_{1,1} + \epsilon^{2}\,\hat{c}_{1,1}, \qquad
\xi(\epsilon G)  = \epsilon\,\hat{s}_{1,2} + \epsilon^{2}\,\hat{c}_{1,2}.
\end{equation}
Thus the self-interaction dominates for $\epsilon\to 0$ unless $\Re \hat{s}_{1,j}=0$.
This holds when $G$ is purely quadratic, and for variational nonlinearities such as those of Section~\ref{sec:app1}, whose direct cubic terms are purely imaginary \cite{sengul2026variational}.
\end{remark}

\section{Applications to Conservative Systems}
\label{sec:app1}
In one space dimension, with unit reference density, no body force and stress $\sigma=W'$, the balance of linear momentum for a deformation $y(x,t)$ in elasticity \cite[pp.~27--30]{ball2002some} reduces to
\begin{equation} \label{system1}
	\begin{aligned}
		& u_t - v_x = 0, \\
		& v_t - \partial_x \sigma(u) = 0,
	\end{aligned}
\end{equation}
with the choice $v = y_t$, $u = y_x$.
Here $u$ is the scalar deformation gradient and $v$ is the velocity.
System~\eqref{system1} is the $p$-system of one-dimensional elasticity, equivalently isentropic gas dynamics in Lagrangian coordinates, written with the stress $\sigma=-p$ and with the roles of the two unknowns interchanged relative to the convention of \cite{smoller1994shock}.
Hyperbolicity is the condition $\sigma'>0$.
The regularization introduced below changes only the linear part and preserves the flux of~\eqref{system1}.
We consider the following regularization to the above system, given by
\begin{equation}\label{system2}
\begin{aligned}
& \partial_t u -\partial_x v = -a \partial_x^4 u, \\
& \partial_t v -\partial_x \sigma (u) = -\delta \partial_x^2 v- \partial_x^4  v,
\end{aligned}
\end{equation}
Here $\delta$ is the bifurcation parameter playing the role of the $\lambda$ in~\eqref{main}.
Thus $m_L=2$ and $r_\lambda=2\le3=2m_L-1$, so Assumption~\ref{cond:D}\ref{D2} is satisfied.
We take a cubic stress--strain law 
\begin{equation}\label{sigma-cubic}
	\sigma(u) = c^2 u + \frac{1}{2} u^2 + \frac{\eta}{3} u^3,
\end{equation}
and assume
\begin{equation} \label{a_c_cond}
	0 < a < c, \quad c^2 > \frac{4a(a+1)^3}{27}.
\end{equation}
The role of the second condition is explained below.

The Hopf bifurcation of~\eqref{system2} at $\eta = 0$ was studied in \cite{yao20142} and generalized by \cite{liyao2015} to arbitrary flux laws and wavenumbers.
The stationary bifurcation with double zero eigenvalue for the same class of systems is analyzed in \cite{wang2018dynamical}.

The nonlinear operator is $G(u,v) = [0,\; uu_x + \eta\,u^2 u_x]^T$ and $m_L=2$ in \eqref{Lij}.
Condition~\ref{cond:D} holds, and the mean-zero restriction fixes the two neutral conserved means described in Remark~\ref{rem:meanzero-necessary}, since $M_0\equiv0$.

\subsection{Linear Analysis}
For $m \in \mathbb{Z}_{\ne 0}$, the matrix $M_m$ given by \eqref{Mm} is
\begin{equation}\label{Mmapp}
M_m= \begin{pmatrix}
-a m^4 & i m \\
c^2 i m  & \delta m^2 - m^4
\end{pmatrix}.
\end{equation}
We let 
\[
\delta_c = a + 1, \quad m_c = 1.
\]
Since $\det M_m = m^2(c^2 + a m^4 (m^2-\delta))$ and $\tr M_m = m^2(\delta - m^2 (a+1))$, the conditions of Assumption~\ref{hopf_assumption} can be checked explicitly:
\begin{enumerate}[label=(\roman*)]
	\item $\tr M_{\pm 1}(\delta_c) = 0$,
	\item For $m=1$, $\det M_1(\delta_c)=c^2-a^2>0$ by \eqref{a_c_cond}.
	For $|m|\ge2$, $\det M_m(\delta_c)=m^2\bigl(am^6-a(a+1)m^4+c^2\bigr)$, and minimizing the cubic $s\mapsto as^3-a(a+1)s^2+c^2$ ($s=m^2$) at its interior critical point $s^\ast=2(a+1)/3$ shows that under the second inequality in \eqref{a_c_cond}, $\det M_m(\delta_c)>0$ for every $m\ne0$.
	\item $\tfrac{d}{d\delta}\tr M_{\pm 1}(\delta) = 1 \ne 0$,
	\item $\tr M_m(\delta_c) < \tr M_{\pm 1}(\delta_c) = 0$ for all $m \ne \pm 1$, since $\tr M_m(\delta_c)=(a+1)m^2(1-m^2)$,
\end{enumerate}

At $\delta = a+1$, with $\omega = \sqrt{c^2 - a^2}$, we have
\[
\beta_{1,1} = i \omega, \quad \beta_{1,2} = -i \omega,
\]
and
\[
\beta_{2,1} = -2\left (3(a+1) + \sqrt{(3 - 5 a )^2 - c^2}\right ), \quad \beta_{2,2} =  -2\left(3 (a+1)-\sqrt{(3- 5 a )^2 - c^2} \right).
\]
For simple eigenvalues, we choose the eigenvectors and adjoint eigenvectors in the form
\begin{equation}\label{application_evecs}
\mathbf{q}_{m,n} = \begin{bmatrix} \dfrac{im}{a m^4 +\beta_{m,n}} \\ 1 \end{bmatrix}, \quad 
\mathbf{q}_{m,n}^* = \kappa_{m,n} \begin{bmatrix} 1 \\[2pt] \dfrac{i\,(a m^4 + \overline{\beta_{m,n}})}{c^2 m} \end{bmatrix}, \quad 
m>0, n=1,2,
\end{equation}
with normalization constants
\begin{equation}\label{kappa_app1}
\overline{\kappa_{m,n}} = \frac{1}{2\pi i} \frac{c^2 m\,(a m^4 + \beta_{m,n})}{c^2 m^2 - (a m^4 + \beta_{m,n})^2},
\end{equation}
fixed by the orthogonality condition~\eqref{orth}.
The critical eigenvectors satisfy $R\,\overline{\mathbf{q}_{1,1}} = -\mathbf{q}_{1,2}$, so the normalization of Lemma~\ref{lem:crit_norm} holds with $\gamma=-1$.
At exceptional parameter values where $M_2$ has a repeated defective eigenvalue, the second-harmonic eigenvectors in~\eqref{application_evecs}--\eqref{kappa_app1} are replaced by the matrix-resolvent formula~\eqref{eq:c_hat_matrix}.
The simplified coefficients below extend regularly to those values.

\subsection{Normal Form Coefficients and Bifurcation Analysis}
The nonzero multi-index coefficients of~\eqref{system2} are
\[
\mathbf{a}_{\alpha_1, \alpha_2} =
\begin{cases}
[0 \quad 1]^T,    & \alpha_1 = \alpha_1^q,\ \alpha_2 = \alpha_2^q, \\
[0 \quad \eta]^T, & \alpha_1 = \alpha_1^c,\ \alpha_2 = \alpha_2^c, \\
[0 \quad 0]^T,    & \text{otherwise,}
\end{cases}
\]
where
\[
\alpha_1^q = (1, 1, 0, \dots), \qquad \alpha_2^q = (0, 0, 0, \dots),
\]
and
\[
\alpha_1^c = (2, 1, 0, \dots), \quad \alpha_2^c = (0, \dots).
\]
Substituting~\eqref{application_evecs} into~\eqref{sigma_n} gives
\begin{equation}\label{sigma1_app1}
\sigma_{1,n}(\alpha_1^q, \alpha_2^q)
= \frac{(-1)^{n+1} a + i\omega}{2\omega}, \qquad n=1,2,
\end{equation}
and from~\eqref{ABalpha}--\eqref{CUV},
\[
C_{U, 1, 1} = \frac{a + i\omega}{2\omega}, \qquad
C_{U, 2, 1} = C_{V, 1, 1} = C_{V, 2, 1} = 0.
\]
Away from the repeated-eigenvalue locus of $M_2$, the scalar eigenbasis formulas give
\begin{equation}\label{sigma2_app1}
\sigma_{2,k}(\alpha_1^q, \alpha_2^q)
= -\frac{i\,(16 a + \beta_{2,k})^2}{4 c^2 - (16 a + \beta_{2,k})^2}, \qquad k=1,2.
\end{equation}
From~\eqref{Phi_coeff}, using $u_{1,1}^2 = -1/(a + i\omega)^2$:
\begin{equation}\label{PhiAA_app1}
\Phi_{AA, 2, k}
= \frac{\sigma_{2,k}\,u_{1,1}^2}{2\beta_{1,1} - \beta_{2,k}}
= \frac{i\,(16 a + \beta_{2,k})^2}{(a + i\omega)^2\,\bigl(4 c^2 - (16 a + \beta_{2,k})^2\bigr)\,(2 i\omega - \beta_{2,k})}.
\end{equation}
For $\Phi_{AB, 2, k}$ the master formula~\eqref{Q_AB} for $(\abs{\alpha_1}, \abs{\alpha_2}) = (2, 0)$ gives $\mathcal{Q}_{\alpha_1^q, \alpha_2^q} = 2\,u_{1,1}\,u_{1,2} = -2/c^2$, hence
\begin{equation}\label{PhiAB_app1}
\Phi_{AB, 2, k}
= -\frac{1}{\beta_{2,k}}\,\sigma_{2,k}(\alpha_1^q, \alpha_2^q) 2\,u_{1,1}\,u_{1,2}
= -\frac{2 i\,(16 a + \beta_{2,k})^2}{c^2\,\beta_{2,k}\,\bigl(4 c^2 - (16 a + \beta_{2,k})^2\bigr)}.
\end{equation}
Substituting into~\eqref{c_hat_k} with $C_{U, 1, 1}\,\overline{u_{1,1}} = -i\,(a + i\omega)^2/(2\omega c^2)$ and $u_{2,k} = 2 i/(16 a + \beta_{2,k})$:
\begin{equation}\label{c_hat_11_app1}
\begin{aligned}
\hat{c}_{1, 1}(a, c)
&= \frac{i}{c^2 \omega }\sum_{k = 1}^{2}\frac{16 a + \beta_{2,k}}{\bigl(4 c^2 - (16 a + \beta_{2,k})^2\bigr)(2 i\omega - \beta_{2,k})}\\
&= \frac{1}{12 c^2\,\bigl[2(a + 1)(a^2 - c^2) + i\omega\,a\,(16 - 5 a)\bigr]}.
\end{aligned}
\end{equation}
Similarly, $C_{U, 1, 1}\,\overline{u_{1,2}} = -i/(2\omega)$, and~\eqref{c_hat_k}~with~\eqref{PhiAB_app1} yields
\begin{equation}\label{c_hat_12_app1}
\hat{c}_{1, 2}(a, c)
= -\frac{2 i}{c^2 \omega}\sum_{k = 1}^{2}\frac{16 a + \beta_{2,k}}{\beta_{2,k}\,\bigl(4 c^2 - (16 a + \beta_{2,k})^2\bigr)}
= -\frac{i}{2\, c^2 \omega\,(c^2 - 16 a(a - 3))}.
\end{equation}

Li and Yao~\cite{liyao2015} generalize the model of Yao~\cite{yao20142} to arbitrary flux laws and wavenumbers, and give explicit closed forms for the cubic coefficients $b_0$ and $c_0$ of that family in their Eqs.~(4.9)--(4.11).
At $k_0=1$ their critical eigenvector is $(\omega-ia)\,\mathbf q_{1,1}$, with the same pairing normalization as~\eqref{orth}.
Their amplitude is therefore $z_1=A/(\omega-ia)$, and by Remark~\ref{rem:norm_invariance} their cubic coefficients are $\abs{\omega-ia}^2=c^2$ times ours.
Specializing their Eqs.~(4.6)--(4.7) to $k_0=1$, $\delta_c=a+1$, $\sigma'(0)=c^2$, $\sigma''(0)=1$ and $\sigma'''(0)=2\eta$ indeed gives $b_0 = c^2\zeta$ and $c_0 = c^2\xi$, with $\zeta,\xi$ as in~\eqref{normal_form_app1}.
Their Eqs.~(4.9) and~(4.11) agree with these values.
So~\eqref{c_hat_11_app1}--\eqref{normal_form_app1} agree with \cite{liyao2015} on the overlap.
In particular, their $\Re c_0 = 0$ is the identity $\xiR\equiv0$ of~\eqref{real_parts_app1}, and their observation that $\sigma'''(0)$ affects only the angular equations is the statement that $\hat{s}_{1,j}$ in~\eqref{s_hat_app1} is purely imaginary.

The cubic nonlinear term $\eta u^2 u_x$ determines the self-interaction coefficients. We have $\mathfrak{m}_t(\alpha_1^c) = 1$ and only $s_{n,1}$ is non-zero in~\eqref{s_k_coeff}.
Applying~\eqref{s_hat_k} with $|u_{1,1}|^2 = |u_{1,2}|^2 = 1/c^2$ yields
\begin{equation}\label{s_hat_app1}
\hat{s}_{1, 1}(a, c, \eta) = \frac{i\eta}{2 c^2 \omega}, \qquad
\hat{s}_{1, 2}(a, c, \eta) = \frac{i\eta}{c^2 \omega}.
\end{equation}
Finally, the normal form coefficients are given by
\begin{equation}\label{normal_form_app1}
\zeta = \hat{c}_{1, 1}(a, c) + \frac{i\eta}{2 c^2 \omega}, \qquad
\xi   = \frac{i}{c^2 \omega}\left[\eta - \frac{1}{2\bigl(c^2 - 16 a(a - 3)\bigr)}\right].
\end{equation}
Using $0<a<c$, their real parts evaluate explicitly to
\begin{equation}\label{real_parts_app1}
\zetaR = -\frac{a + 1}{6 c^2\,\bigl[a^2 (5 a - 16)^2 + 4 (a + 1)^2 (c^2 - a^2)\bigr]} < 0,
\qquad \xiR \equiv 0.
\end{equation}
The system therefore lies in Region~IV: traveling waves are saddles and standing waves are stable.
The identity $\xiR\equiv0$ holds for every admissible $(a,c,\eta)$.
By~\eqref{s_hat_app1} the cubic stress coefficient $\eta$ enters $\zeta$ and $\xi$ only through their imaginary parts, so the displayed $\zetaR$ is determined by the quadratic stress term alone.

\subsection{A Conservative Bilinear Family Realizing the Full Classification}\label{sec:app1_TW}

It is natural to attribute the identity $\xiR\equiv0$ of the previous application, which holds at every admissible parameter value, to the conservativity of its nonlinearity.
Conservativity alone is not sufficient.
Keeping the nonlinearity conservative and enlarging it to the general first-order bilinear family already permits all six regions of Table~\ref{table:signs}, so the cancellation in~\eqref{real_parts_app1} must reflect a finer property of the stress--strain flux.
That property is the variational form of the nonlinearity \cite{sengul2026variational}.
We keep the same linear part~\eqref{system2}, so Assumption~\ref{cond:D}\ref{D2} remains satisfied, and take the most general conservative $O(2)$-equivariant bilinear nonlinearity whose derivatives have order at most one,
\begin{equation}\label{G_bilinear}
G(u,v)=\begin{bmatrix} g_{11}\,(u_x v + u v_x)\\[3pt] g_{21}\,u u_x + g_{22}\,v v_x \end{bmatrix},
\qquad g_{11},g_{21},g_{22}\in\mathbb{R}.
\end{equation}
The quadratic part of the stress--strain model~\eqref{sigma-cubic} is obtained when $(g_{11},g_{21},g_{22})=(0,1,0)$.

The nonzero multi-index coefficients of~\eqref{G_bilinear} are
\[
\mathbf{a}_{\alpha_1,\alpha_2} =
\begin{cases}
[\,g_{11}\ \ 0\,]^T, & (\alpha_1,\alpha_2)=\bigl((0,1,0,\dots),(1,0,\dots)\bigr)
                       \text{ and } \bigl((1,0,\dots),(0,1,0,\dots)\bigr),\\
[\,0\ \ g_{21}\,]^T, & \alpha_1=(1,1,0,\dots),\ \alpha_2=0,\\
[\,0\ \ g_{22}\,]^T, & \alpha_1=0,\ \alpha_2=(1,1,0,\dots),
\end{cases}
\]
and the linear data (the eigenvalues $\beta_{m,n}$ and eigenvectors~\eqref{application_evecs}) is unchanged.

Since the nonlinearity is purely quadratic, $\hat s_{1,j}=0$ and both normal form coefficients are given by the cross-interaction formulas of Theorem~\ref{thm:main}.
The closed forms below were verified numerically with the companion package.
Writing $D_\zeta=a^{2}(5a-16)^{2}+4(a+1)^{2}(c^{2}-a^{2})$ and $D_\xi=c^{2}-16a(a-3)$, the real parts are
\begin{equation}\label{bilinear_real}
\zetaR=\frac{\mathbf g^{\!\top} N\,\mathbf g}{12\,c^{2}D_\zeta},
\qquad
\xiR=-\,\frac{a\,(g_{11}-g_{22})\bigl[c^{2}(g_{11}-4g_{22})+4g_{21}\bigr]}{c^{2}\,D_\xi},
\qquad \mathbf g=(g_{11},g_{21},g_{22})^{\!\top},
\end{equation}
with $N=N^{\!\top}$ given by
\begin{equation}\label{bilinear_N}
\begin{aligned}
N_{11}&=-2c^{2}\bigl(9a^{3}-72a^{2}+4ac^{2}+108a+4c^{2}\bigr), &
N_{22}&=-2(a+1),\\
N_{33}&=c^{2}\bigl(63a^{3}-126a^{2}-2ac^{2}-2c^{2}\bigr), &
N_{12}&=\tfrac12\bigl(a^{3}+22a^{2}-8ac^{2}-168a-8c^{2}\bigr),\\
N_{13}&=\tfrac12 c^{2}\bigl(63a^{3}-126a^{2}-8ac^{2}-8c^{2}\bigr), &
N_{23}&=\tfrac12\bigl(35a^{3}-154a^{2}-4ac^{2}-4c^{2}\bigr).
\end{aligned}
\end{equation}
Setting $g_{11}=g_{22}=0$ returns $\zetaR=-(a+1)/(6c^{2}D_\zeta)\,g_{21}^{2}$ and $\xiR\equiv0$, recovering~\eqref{real_parts_app1}.
The coefficients $\zeta,\xi$ are complex.
Their imaginary parts play no role in the bifurcation classification, and are reproduced by the companion package.
Already the slice $a=\tfrac12$, $c=1$, $g_{21}=1$ intersects all six regions of Table~\ref{table:signs}, as shown in Figure~\ref{fig:app1_phase}.
The stress--strain model is the origin in Region~IV.
On this slice, setting $g_{11}=0$ leaves only Regions~IV--VI, and no two terms in~\eqref{G_bilinear} suffice to realize the full classification.

Direct simulations compare the stress--strain point $(g_{11},g_{22})=(0,0)$, in Region~IV, with a point in Region~V.
Figure~\ref{fig:app1_sim} shows convergence to the predicted standing and traveling waves, respectively.
For the traveling branch, the normal form predicts
\[
\abs{\hat u_1}
=\abs{u_{m_c,1}}\sqrt{-\Re\beta_{m_c,1}/\zetaR}.
\]
After the fixed discrete-Fourier normalization, the simulated saturation follows this prediction, testing the magnitude of $\zetaR$ as well as its sign.
Regions~I, II, III, and~VI have no attracting small-amplitude branch and hence no analogous direct check.

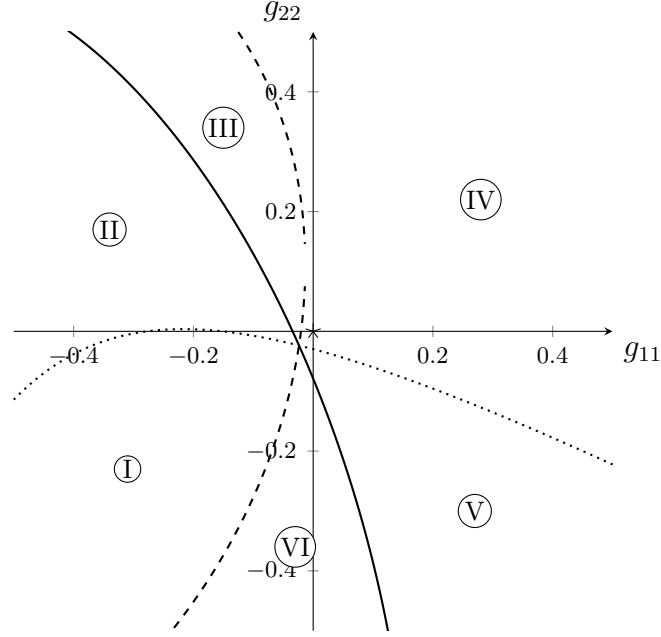
\begin{figure}[tb]
\centering
\begin{tikzpicture}
\begin{axis}[
  width=7.2cm, height=7.2cm,
  xmin=-0.5, xmax=0.5, ymin=-0.5, ymax=0.5,
  axis equal image,
  xlabel={$g_{11}$}, ylabel={$g_{22}$},
  ylabel style={rotate=-90},
  xtick={-0.4,0,0.4}, ytick={-0.4,0,0.4},
  extra x ticks={0}, extra x tick style={grid=major, grid style={gray!40}, tick label style={opacity=0}},
  extra y ticks={0}, extra y tick style={grid=major, grid style={gray!40}, tick label style={opacity=0}},
  tick label style={font=\scriptsize},
  clip=true, samples=400, restrict y to domain*=-0.55:0.55, unbounded coords=jump,
]
\addplot[black,thick,domain=-0.5:0.5]({x},{(-(285*x+321)+sqrt((285*x+321)^2-4*213*(690*x^2+723*x+24)))/426});
\addplot[black,thick,domain=-0.5:0.5]({x},{(-(285*x+321)-sqrt((285*x+321)^2-4*213*(690*x^2+723*x+24)))/426});
\addplot[black,thick,dashed,domain=-0.5:0.5]({x},{(6570*x+3303+sqrt((6570*x+3303)^2-4*14517*(17001*x^2+25227*x+504)))/29034});
\addplot[black,thick,dashed,domain=-0.5:0.5]({x},{(6570*x+3303-sqrt((6570*x+3303)^2-4*14517*(17001*x^2+25227*x+504)))/29034});
\addplot[black,thick,dotted,domain=-0.5:0.5]({x},{(-(18540*x+16785)+sqrt((18540*x+16785)^2+4*5571*(11979*x^2+5139*x+504)))/(-11142)});
\addplot[black,thick,dotted,domain=-0.5:0.5]({x},{(-(18540*x+16785)-sqrt((18540*x+16785)^2+4*5571*(11979*x^2+5139*x+504)))/(-11142)});
\node[circle,draw,fill=white,inner sep=0.8pt,font=\footnotesize] at (axis cs:-0.31,-0.23){I};
\node[circle,draw,fill=white,inner sep=0.8pt,font=\footnotesize] at (axis cs:-0.34, 0.17){II};
\node[circle,draw,fill=white,inner sep=0.8pt,font=\footnotesize] at (axis cs:-0.15, 0.34){III};
\node[circle,draw,fill=white,inner sep=0.8pt,font=\footnotesize] at (axis cs: 0.28, 0.22){IV};
\node[circle,draw,fill=white,inner sep=0.8pt,font=\footnotesize] at (axis cs: 0.27,-0.30){V};
\node[circle,draw,fill=white,inner sep=0.8pt,font=\footnotesize] at (axis cs:-0.03,-0.36){VI};
\addplot[only marks,mark=star,mark size=3pt,black]coordinates{(0,0)};
\end{axis}
\end{tikzpicture}
\caption{Wave-selection regions of the conservative bilinear family~\eqref{G_bilinear} in the $(g_{11},g_{22})$-plane, at $a=\tfrac12$, $c=1$, $g_{21}=1$ (the linear operator~\eqref{system2}, $m_c=1$).
All six regions of Table~\ref{table:signs} occur.
The stress--strain model~\eqref{sigma-cubic} with $\eta=0$ is at the origin (Region~IV).
Solid, dashed and dotted curves are $\zetaR=0$, $\zetaR+\xiR=0$ and $\zetaR-\xiR=0$.}
\label{fig:app1_phase}
\end{figure}
\begin{figure}[tb]
\centering
\includegraphics[width=0.85\linewidth]{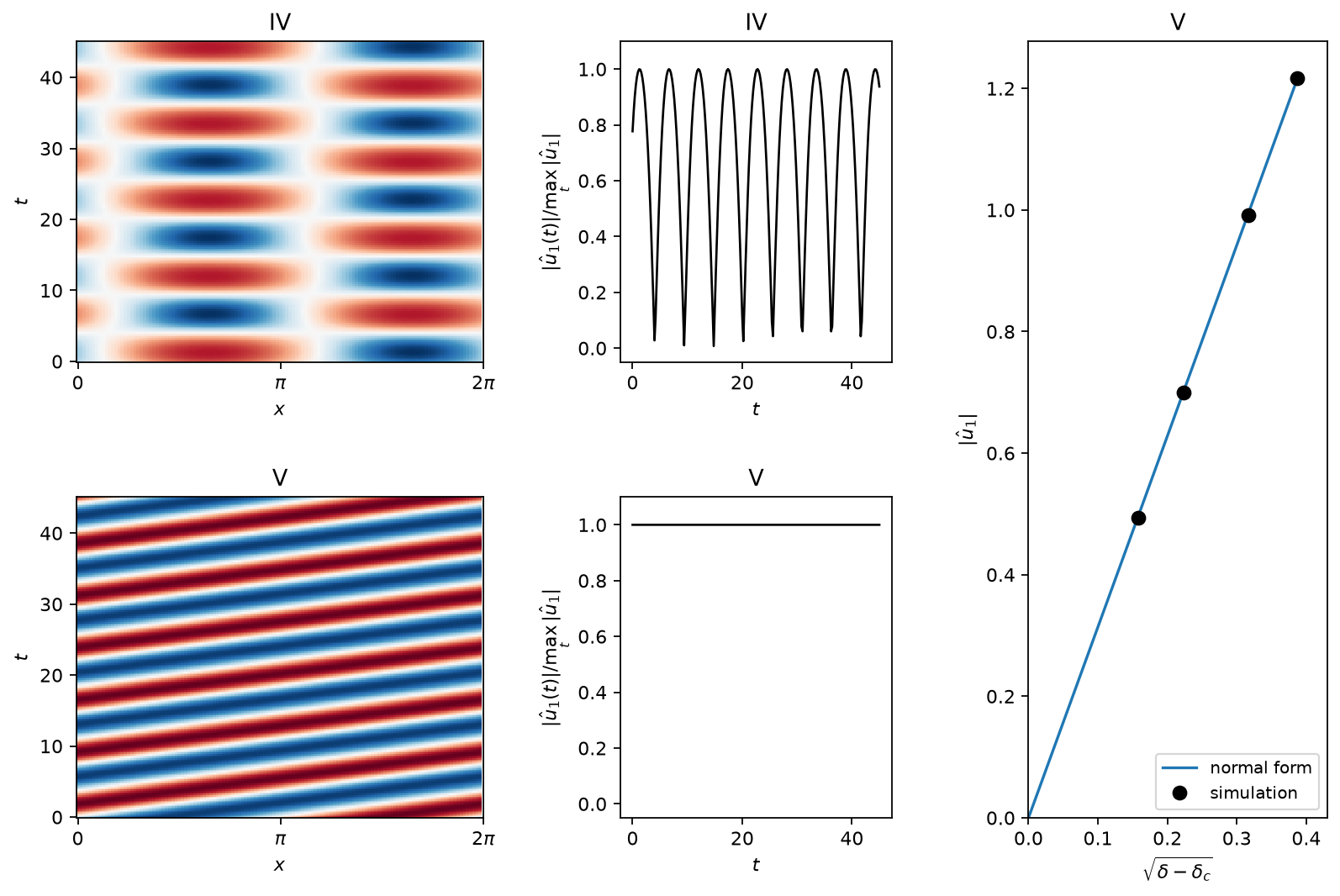}
\caption{Direct integration of~\eqref{system2} at $a=\tfrac12$, $c=1$, $g_{21}=1$.
Left and middle, $\delta-\delta_c=0.05$.
Top, Region~IV ($g_{11}=g_{22}=0$, the stress--strain model): $u(x,t)$ has fixed nodes and $|\hat u_1(t)|$ pulses to zero, a standing wave.
Bottom, Region~V ($g_{11}=0.35$, $g_{22}=-0.30$): $u(x,t)$ translates and $|\hat u_1(t)|$ is constant, a traveling wave.
Right, Region~V: the saturated $|\hat u_1|$ at several $\delta-\delta_c>0$ against the normal form branch, which is proportional to $(\delta-\delta_c)^{1/2}$.}
\label{fig:app1_sim}
\end{figure}

\FloatBarrier
\section{Proofs of the Main Results}\label{sec:proof}

\subsection{Existence of the Center Manifold}\label{sec:CM_existence}

We reduce the dynamics of \eqref{main} near $\lambda=\lambda_c$ to a four-dimensional
center manifold by applying the center manifold theorems of \cite[Chapter~2]{haragus2010local}, specifically Theorems~2.9 and~2.17 together with the parameter-dependent Theorem~3.3
(see also \cite{vanderbauwhede1992center}).
We write~\eqref{main} as $\dot U=L_{\lambda_c}U+\mathcal R(U,\lambda)$ with $\mathcal R(U,\lambda):=(L_\lambda-L_{\lambda_c})U+G(U)$.
We verify Hypotheses~2.1 (in its parameter-dependent form, Hypothesis~3.1), 2.4, and~2.15 of \cite{haragus2010local} in the
Hilbert triple $\mathcal{Z}=\mathcal{H}^{2m_L}\hookrightarrow
\mathcal{Y}=\mathcal{H}^{2m_L-r}\hookrightarrow\mathcal{X}=\mathcal{H}^{0}$
with $r=\max(r_G,r_\lambda)$.
Condition~\eqref{rGsL} and Assumption~\ref{cond:D}\ref{D2} give $2m_L-r\ge 1$, so these embeddings are continuous and dense.
By \cite[Theorem~2.17]{haragus2010local}, Hypothesis~2.15 implies
Hypothesis~2.7, and Theorem~2.9 then applies.

\medskip
\noindent\textbf{Hypothesis 2.1.}
By \eqref{Lij}, $L\in\mathcal{L}(\mathcal{Z},\mathcal{X})$.
Since $2m_L-r_G\ge 1>\tfrac12$, the space $H^{2m_L-r_G}_{\mathrm{per}}$ is a Banach algebra, and the order bound on every factor in~\eqref{G} gives $G\in C^\infty\bigl(\mathcal{Z},\bigl(H^{2m_L-r_G}_{\mathrm{per}}\bigr)^2\bigr)$ with $G(0)=0$ and $DG(0)=0$.
Conservativity ensures that each component of $G$ has zero mean, so the image lies in $\mathcal{Y}$.
Since the coefficients are smooth in $\lambda$ and every entry whose coefficient varies has order at most $r_\lambda$, the difference $L_\lambda-L_{\lambda_c}$ lies in $\mathcal L(\mathcal Z,\mathcal Y)$, depends smoothly on $\lambda$, and vanishes at $\lambda_c$.
Hence $\mathcal R\in C^\infty(\mathcal Z\times\mathbb R,\mathcal Y)$ with $\mathcal R(0,\lambda_c)=0$ and $D_U\mathcal R(0,\lambda_c)=0$, which is Hypothesis~3.1 of \cite{haragus2010local} with $L=L_{\lambda_c}$ and parameter $\mu=\lambda-\lambda_c$.

\medskip
\noindent\textbf{Hypothesis 2.4.}
By Lemma~\ref{lem:PES}, $\sigma(L(\lambda_c))=\sigma_0\cup\sigma_-$ with $\sigma_0=\{\pm i\omega\}$ semisimple of multiplicity $2$ and $\sup_{\sigma_-}\Re\lambda\le-\rho<0$.
\medskip

\noindent\textbf{Hypothesis 2.15 (Resolvent estimate).}
Hypothesis~2.15 consists of the two estimates~\cite{kato1966perturbation,pazy1983semigroups} $\|(i\tilde\omega-L)^{-1}\|_{\mathcal L(\mathcal X)}\le c/|\tilde\omega|$ and $\|(i\tilde\omega-L)^{-1}\|_{\mathcal L(\mathcal Y,\mathcal Z)}\le c/|\tilde\omega|^{1-\alpha}$ for $|\tilde\omega|\ge\omega_0$, with some $\alpha\in[0,1)$.
In Hilbert spaces the first estimate alone gives the conclusion of Theorem~2.9 \cite[Remark~2.16]{haragus2010local}.
Theorems~3.3 and~3.13 and Hypothesis~3.20 of \cite{haragus2010local}, used below, rest on Hypothesis~2.7, so we verify both estimates.
The first is the content of the following theorem, and the second follows from its proof.

\begin{theorem}[Resolvent estimate]\label{thm:resolvent}
Suppose Assumption~\ref{cond:D}\ref{D1} and Assumption~\ref{hopf_assumption} hold.
Then there exist constants $\omega_0, C > 0$, depending only on the coefficients of $L$, such that for every $\tilde\omega\in\mathbb{R}$ with $|\tilde\omega|\ge\omega_0$, $i\tilde\omega$ lies in the resolvent set of $L$, and
\begin{equation}\label{eq:resolvent-bound}
\|(i\tilde\omega\,\mathrm{Id} - L)^{-1}\|_{\mathcal{L}(L^2)} \le \frac{C}{|\tilde\omega|}.
\end{equation}
\end{theorem}

\begin{proof}
The proof uses the adjugate formula for $2\times 2$ matrices and a partition of the $(\tilde\omega,m)$-plane into three regimes.
This strategy is a self-contained instance of the parameter ellipticity technique for Fourier multiplier systems (see \cite{denkhieberprus2003} for the general $n\times n$ parabolic setting, and \cite{agmon1962} for the scalar elliptic case).

Write $\tau:=\tr M_m(\lambda_c)$ and $\delta:=\det M_m(\lambda_c)$.
Both are real, and the diagonal entries $M_{11},M_{22}$ of $M_m(\lambda_c)$ are real while the off-diagonal entries $M_{12},M_{21}$ are purely imaginary.
Set $p_m(z):=z^2-\tau z+\delta=\det(zI_2-M_m(\lambda_c))$.

\emph{Step 1: Reduction to the symbol.}
Since $L$ acts on $\mathcal{H}$ as the Fourier multiplier $e^{imx}\mathbf c\mapsto e^{imx}M_m\mathbf c$, the operator $i\tilde\omega\,\mathrm{Id}-L$ is invertible iff every $i\tilde\omega I_2-M_m$ with $m\ne0$ is, and
\[
\|(i\tilde\omega\,\mathrm{Id}-L)^{-1}\|_{\mathcal L(L^2)}=\sup_{m\ne0}\|(i\tilde\omega I_2-M_m)^{-1}\|_2 .
\]

\emph{Step 2: Adjugate bound.}
Whenever $p_m(i\tilde\omega)\ne0$, since $p_m(i\tilde\omega)=(\delta-\tilde\omega^2)-i\tau\tilde\omega$,
\[
\|(i\tilde\omega I_2-M_m)^{-1}\|_2\le\frac{2N}{D},\qquad
D:=\sqrt{(\delta-\tilde\omega^2)^2+\tau^2\tilde\omega^2},
\]
\[
N:=\max\bigl\{\sqrt{M_{11}^2+\tilde\omega^2},\ \sqrt{M_{22}^2+\tilde\omega^2},\ |M_{12}|,\ |M_{21}|\bigr\},
\]
using $(i\tilde\omega I_2-M_m)^{-1}=p_m(i\tilde\omega)^{-1}\operatorname{adj}(i\tilde\omega I_2-M_m)$ and $\|\operatorname{adj}\|_2\le\|\operatorname{adj}\|_F\le2N$.

\emph{Step 3: Coefficient bounds.}
By Assumption~\ref{cond:D}\ref{D1} the off-diagonal entries have degree $\le2m_L-1$ and each diagonal entry has degree $2m_L$ in $m$, so for $|m|\ge1$
\[
|M_{jj}(m)|\le A\,m^{2m_L},\qquad |M_{12}(m)|,|M_{21}(m)|\le B\,m^{2m_L},
\]
with $A,B$ depending only on the coefficients of $L$.
Hence $N\le A_1(m^{2m_L}+|\tilde\omega|)$, $A_1:=A+B+1$.

\emph{Step 4: Leading behavior.}
By Assumption~\ref{cond:D}\ref{D1}, $\deg(M_{12}M_{21})\le4m_L-2$, so $\delta$ has degree $4m_L$ with leading coefficient $b_{1,2m_L}b_{2,2m_L}$.
Condition~(ii) of Assumption~\ref{hopf_assumption} ($\delta>0$ for large $m$) forces $b_{1,2m_L}b_{2,2m_L}>0$.
Thus $b_{1,2m_L}$ and $b_{2,2m_L}$ share a sign, $b_{1,2m_L}+b_{2,2m_L}\ne0$, and $\tau$ has degree exactly $2m_L$.
Consequently there are constants $0<c_1\le C_1$, $0<c_2\le C_2$ and an integer $m_*\ge1$, depending only on the coefficients of $L$ and chosen larger than every real root of $\tau$ such that
\[
c_1m^{4m_L}\le\delta(m)\le C_1m^{4m_L},\qquad c_2m^{2m_L}\le|\tau(m)|\le C_2m^{2m_L}\qquad(|m|\ge m_*).
\]
In particular $m_*>m_c$, since $\tau(m_c)=0$.

\emph{Step 5: High modes $|m|\ge m_*$.}
Fix $|m|\ge m_*$.
Then $m\ne\pm m_c$, so $\tau\ne0$ and $i\tilde\omega\notin\sigma(M_m)$ for every $\tilde\omega\in\mathbb{R}$.
The three cases below partition all $\tilde\omega\in\mathbb{R}$, and in each $D>0$.

\emph{(a) $\tilde\omega^2\le\tfrac14c_1m^{4m_L}$.}
Then $\tilde\omega^2\le\tfrac14\delta$, so $D\ge\delta-\tilde\omega^2\ge\tfrac34c_1m^{4m_L}$, while $|\tilde\omega|\le\tfrac12\sqrt{c_1}\,m^{2m_L}$ gives $N\le A_1(1+\tfrac12\sqrt{c_1})m^{2m_L}$.
Hence $|\tilde\omega|N/D\le\tfrac{2A_1(1+\frac12\sqrt{c_1})}{3\sqrt{c_1}}=:K_a$.

\emph{(b) $\tilde\omega^2>\tfrac14c_1m^{4m_L}$ and $\tilde\omega^2\ge2\delta$.}
Then $D\ge\tilde\omega^2-\delta\ge\tfrac12\tilde\omega^2$, and $m^{2m_L}<\tfrac{2}{\sqrt{c_1}}|\tilde\omega|$ gives $N\le A_2|\tilde\omega|$ with $A_2:=A_1(1+\tfrac2{\sqrt{c_1}})$.
Hence $|\tilde\omega|N/D\le2A_2=:K_b$.

\emph{(c) $\tilde\omega^2>\tfrac14c_1m^{4m_L}$ and $\tilde\omega^2<2\delta$.}
Then $|\tilde\omega|<\sqrt{2C_1}\,m^{2m_L}$, so $|\tau|\ge c_2m^{2m_L}\ge\tfrac{c_2}{\sqrt{2C_1}}|\tilde\omega|$ and $D\ge|\tau\tilde\omega|\ge\tfrac{c_2}{\sqrt{2C_1}}\tilde\omega^2$, while $N\le A_2|\tilde\omega|$ as in~(b).
Hence $|\tilde\omega|N/D\le\tfrac{A_2\sqrt{2C_1}}{c_2}=:K_c$.

In every case $D>0$, so $i\tilde\omega\notin\sigma(M_m)$, and $\|(i\tilde\omega I_2-M_m)^{-1}\|_2\le2N/D\le2K/|\tilde\omega|$ with $K:=\max(K_a,K_b,K_c)$.

\emph{Step 6: Low modes $1\le|m|<m_*$.}
For the finitely many such $m$ (they include the critical pair $\pm m_c$), set $\omega_0:=2\max_{1\le|m|<m_*}\|M_m\|_2$.
Since $\omega=|\beta_1(\lambda_c)|\le\|M_{m_c}\|_2$, we have $\omega_0\ge2\|M_{m_c}\|_2>\omega$, so $i\omega$ is excluded from the range below.
For $|\tilde\omega|\ge\omega_0$ one has $\|M_m\|_2\le\tfrac12|\tilde\omega|$, so a Neumann series gives $i\tilde\omega I_2-M_m$ invertible with $\|(i\tilde\omega I_2-M_m)^{-1}\|_2\le2/|\tilde\omega|$.

For $|\tilde\omega|\ge\omega_0$ every $i\tilde\omega I_2-M_m$ ($m\ne0$) is invertible, so $i\tilde\omega\in\rho(L)$, and by Steps~5--6 the supremum in Step~1 is at most $C/|\tilde\omega|$ with $C:=\max(2K,2)$.
\end{proof}

The second estimate of Hypothesis~2.15 follows from the same proof.
For $|\tilde\omega|\ge\omega_0$, case~(a) of Step~5 gives $\|(i\tilde\omega I_2-M_m)^{-1}\|_2\lesssim m^{-2m_L}$ with $|\tilde\omega|\lesssim m^{2m_L}$, while cases~(b), (c) and Step~6 give $\|(i\tilde\omega I_2-M_m)^{-1}\|_2\lesssim|\tilde\omega|^{-1}$ with $m^{2m_L}\lesssim|\tilde\omega|$.
Hence $\|(i\tilde\omega I_2-M_m)^{-1}\|_2\le C'/(|\tilde\omega|+|m|^{2m_L})$ for all $m\ne0$, with $C'$ depending only on the coefficients of $L$.
The norms of $\mathcal Z$ and $\mathcal Y$ are equivalent to the $\ell^2$ norms of $|m|^{2m_L}\widehat U_m$ and $|m|^{2m_L-r}\widehat F_m$, and $|m|^r\le(|\tilde\omega|+|m|^{2m_L})^{r/(2m_L)}$.
Therefore $\|(i\tilde\omega-L)^{-1}\|_{\mathcal L(\mathcal Y,\mathcal Z)}\le C'|\tilde\omega|^{-(1-\alpha)}$ with $\alpha=r/(2m_L)\in[0,1)$, and Hypothesis~2.15 holds.

\begin{remark}\label{rem:resolvent-hyp}
The proof uses only Assumption~\ref{cond:D}\ref{D1} and condition~(ii) of Assumption~\ref{hopf_assumption}.
Conditions~(i), (iii) and~(iv) play no role in it.
\end{remark}

Hypotheses~2.1, 2.4, and~2.15 of \cite{haragus2010local} are verified, and the existence assertions in Theorem~\ref{thm:cm_reduction} follow from \cite[Theorems~2.9, 2.17 and~3.3]{haragus2010local}.
The $O(2)$-equivariance of the vector field allows the center manifold to be chosen $O(2)$-equivariant by \cite[Theorem~3.13]{haragus2010local}.
Moreover, at $\lambda=\lambda_c$, Assumption~\ref{hopf_assumption} leaves no spectrum in the open right half-plane, and the resolvent estimate above gives Hypothesis~3.20 by \cite[Remark~B.2]{haragus2010local}.
Local attraction then follows from \cite[Theorem~3.23]{haragus2010local}.

\subsection{Reduced System for the Hopf Bifurcation}\label{sec:reduced_system}

We decompose the solution on the center manifold
as $\psi = \psi_c + \Phi$, where $\Phi$ is the stable part and
\begin{equation}\label{psi_c}
\psi_c = z(t)\,e^{im_c x} + \overline{z(t)}\,e^{-im_c x},
\end{equation}
is the center part, with $z(t) = A(t)\mathbf{q}_{m_c,1} + B(t)\mathbf{q}_{m_c,2}$ the Fourier coefficient of $\psi_c$ at mode $m_c$.
We write
\begin{equation}\label{z12ab}
z(t) = (z_1(t), z_2(t))^\top \quad 
z_1(t) = A(t) u_{m_c,1} + B(t) u_{m_c,2}, \quad 
z_2(t) = A(t) v_{m_c,1} + B(t) v_{m_c,2}.
\end{equation}
The change of coordinates~\eqref{z12ab} is invertible, since $\mathbf q_{m_c,1} = (u_{m_c,1}, v_{m_c,1})^\top$ and $\mathbf q_{m_c,2} = (u_{m_c,2}, v_{m_c,2})^\top$ are linearly independent.

Thus the components of $\psi_c = (u_c, v_c)^\top$ are
\[
u_c = z_1(t)\,e^{im_cx} + \overline{z_1(t)}\,e^{-im_cx}, \qquad
v_c = z_2(t)\,e^{im_cx} + \overline{z_2(t)}\,e^{-im_cx}.
\]
The center manifold function is expanded directly in Fourier blocks,
\begin{equation}\label{phi}
\Phi=\Phi(\psi_c)=
\sum_{\substack{m>0\\m\ne m_c}}e^{imx}
\begin{pmatrix}U_m(t)\\V_m(t)\end{pmatrix}
+\mathrm{c.c.}
=:\begin{pmatrix}\Phi_1\\\Phi_2\end{pmatrix}.
\end{equation}
This blockwise representation does not assume that the stable matrices $M_m$ are diagonalizable.
Projecting \eqref{main} onto the critical adjoint eigenvectors 
$e_{m_c,n}^*$ and using the orthonormality relation \eqref{orth} gives
\begin{equation}\label{dd}
\frac{dA}{dt} = \beta_{m_c,1} A(t) + \langle G(\psi), e_{m_c,1}^* \rangle_{L^2}, \qquad
\frac{dB}{dt} = \beta_{m_c,2} B(t) + \langle G(\psi), e_{m_c,2}^* \rangle_{L^2}.
\end{equation}
It remains to compute $\mathcal{G}_n := \langle G(\psi_c + \Phi), e_{m_c,n}^* \rangle_{L^2}$ for $n = 1,2$.
Expanding $G(\psi_c+\Phi)$ to first order in $\Phi$ gives $\mathcal{G}_n = \mathcal{S}_n + \mathcal{C}_n + O(\Phi^2)$, where the self-interaction part is
\begin{equation} \label{S_n}
\mathcal{S}_n = \sum_{\abs{\alpha_1}+\abs{\alpha_2}\geq 2} 
 \int_{0}^{2\pi} 
(Du_c)^{\alpha_1}(Dv_c)^{\alpha_2}\, \langle\mathbf{a}_{\alpha_1,\alpha_2}, e_{m_c,n}^*\rangle_{\mathbb{C}^2}\,dx,
\end{equation}
and the cross-interaction part, linear in $\Phi$, is
\begin{equation}\label{C_n}
\begin{split}
\mathcal{C}_n = \sum_{\abs{\alpha_1} + \abs{\alpha_2} \geq 2} \int_{0}^{2\pi}  \Biggl(
&\left[\frac{d}{d\varepsilon}\bigg|_{\varepsilon=0}(D(u_c+\varepsilon\Phi_1))^{\alpha_1}\right](Dv_c)^{\alpha_2} \\
&+ \left[\frac{d}{d\varepsilon}\bigg|_{\varepsilon=0}(D(v_c+\varepsilon\Phi_2))^{\alpha_2}\right](Du_c)^{\alpha_1}
\Biggr) \langle\mathbf{a}_{\alpha_1, \alpha_2},e_{m_c,n}^*\rangle_{\mathbb{C}^2} \, dx,
\end{split}
\end{equation}
with the linearization
\begin{equation}\label{product_rule}
\frac{d}{d\varepsilon}\bigg|_{\varepsilon=0}(D(u_c+\varepsilon\Phi_1))^{\alpha_1}
= \sum_{\substack{j\geq 0\\\alpha_{1,j}\neq 0}}
\alpha_{1,j}\,(\partial_x^j u_c)^{\alpha_{1,j}-1}
\prod_{l\neq j}(\partial_x^l u_c)^{\alpha_{1,l}}\,\partial_x^j\Phi_1,
\end{equation}
and analogously for the $v$-component.

\subsection{Self-Interactions of the Critical Modes}\label{sec:self-interaction}

Inserting the expressions for $u_c$ and $v_c$ into $\mathcal{S}_n$ 
defined in \eqref{S_n}, we have $\mathcal{S}_n = \sum_{N\geq 2} 
\mathcal{S}_{n,N}$, where the order-$N$ contribution is
\begin{equation}\label{S_i}
\mathcal{S}_{n,N} = \sum_{\abs{\alpha_1}+\abs{\alpha_2}=N} 
\int_{0}^{2\pi} \left[ D\left( z_1 e^{im_cx} + \mathrm{c.c.} \right) 
\right]^{\alpha_1} 
\left[ D\left( z_2 e^{im_cx} + \mathrm{c.c.} \right) \right]^{\alpha_2}
\langle\mathbf{a}_{\alpha_1,\alpha_2}, e_{m_c,n}^*\rangle_{\mathbb{C}^2}\, dx,
\end{equation}
for $n=1,2$ and $N \geq 2$.

Expanding by the binomial theorem,
\begin{equation}\label{binomial_expansion}
\begin{aligned}
\mathcal{S}_{n,N} &= \sum_{\abs{\alpha_1}+\abs{\alpha_2} = N} 
\sum_{\substack{0\leq p\leq \mathfrak{m}_e(\alpha_1) \\ 0\leq q\leq \mathfrak{m}_o(\alpha_1) \\ 
0\leq r\leq \mathfrak{m}_e(\alpha_2) \\ 0\leq s\leq \mathfrak{m}_o(\alpha_2)}} 
\binom{\mathfrak{m}_e(\alpha_1)}{p}\binom{\mathfrak{m}_o(\alpha_1)}{q}\binom{\mathfrak{m}_e(\alpha_2)}{r}\binom{\mathfrak{m}_o(\alpha_2)}{s}
(-1)^{q+s}(im_c)^{\mathfrak{m}_t(\alpha_1+\alpha_2)} \\
&\quad \times z_1^{\abs{\alpha_1}-(p+q)} \overline{z}_1^{p+q}
z_2^{\abs{\alpha_2}-(r+s)} \overline{z}_2^{r+s}\,
\langle\mathbf{a}_{\alpha_1, \alpha_2},\mathbf{q}_{m_c,n}^*\rangle_{\mathbb{C}^2}
\int_{0}^{2\pi} e^{iKm_c x}\, dx,
\end{aligned}
\end{equation}
where $K = N - 2(p+q+r+s) - 1$. Since $\int_0^{2\pi} 
e^{iKm_cx}\,dx = 2\pi\delta_{K,0}$, only terms with $K=0$ contribute. 
When $N$ is even, $K=0$ requires $N - 2(p+q+r+s) = 1$, which has no 
solution in non-negative integers, so $\mathcal{S}_{n,N} = 0$.
That proves the following parity result.
\begin{lemma}\label{lem:parity}
Self-interaction of the critical modes through even-order nonlinear 
terms vanishes identically: $\mathcal{S}_{n,N} = 0$ for all even 
$N \geq 2$ and $n=1,2$. Consequently, $\mathcal{S}_n = 
\sum_{N \geq 3,\,N\text{ odd}} \mathcal{S}_{n,N}$.
\end{lemma}
This also follows from symmetry: translation by $\pi/m_c$ acts on the critical amplitudes as $(A,B)\mapsto(-A,-B)$, so the equivariant reduced vector field is odd.
The reflection plays a different role: it relates the coefficients of the two amplitude equations (Section~\ref{sec:normal_form}).
By Lemma~\ref{lem:parity} it suffices to consider odd $N$.
The leading term is $N=3$.
When $N=3$, $K=0$ in \eqref{binomial_expansion} gives $p+q+r+s = 1$, yielding
\begin{equation*}
\mathcal{S}_{n,3}(z_1, z_2) = s_{n,1} z_1 \abs{z_1}^2 + s_{n,2} z_1 \abs{z_2}^2 + s_{n,3} z_2 \abs{z_1}^2 + s_{n,4} z_2 \abs{z_2}^2 + s_{n,5} \overline{z_1}z_2^2 + s_{n,6} z_1^2 \overline{z_2},
\end{equation*}
where the coefficients $s_{n,k}$ are given by \eqref{s_k_coeff}.
Substituting the coordinate change \eqref{z12ab} into $\mathcal{S}_{n,3}$ and inserting into~\eqref{dd} gives, at cubic order and ignoring the cross-interaction $\mathcal{C}_n$,
\begin{equation}\label{red1}
\begin{aligned}
\frac{dA}{dt} &= \beta_{m_c,1} A + P_1(A, B, \bar{A}, \bar{B}), \\
\frac{dB}{dt} &= \beta_{m_c,2} B + P_2(A, B, \bar{A}, \bar{B}),
\end{aligned}
\end{equation}
where $P_n = \hat{s}_{n,1} A|A|^2 + \hat{s}_{n,2} A|B|^2 + \hat{s}_{n,3} B|A|^2 + \hat{s}_{n,4} B|B|^2 + \hat{s}_{n,5} \bar{A}B^2 + \hat{s}_{n,6} A^2\bar{B}$.
These coefficients are obtained by expanding~\eqref{S_i}.
The two resonant coefficients in the $A$ equation are precisely $\hat s_{1,1}$ and $\hat s_{1,2}$ in~\eqref{s_hat_k}.

\subsection{Cross-Interaction Terms}\label{sec:cross-interaction}
Write $\mathcal{C}_n = \sum_{N\geq 2}\mathcal{C}_{n,N}$, where $\mathcal{C}_{n,N}$ collects the terms of~\eqref{C_n} with $\abs{\alpha_1}+\abs{\alpha_2}=N$.
Since $\Phi=O(\abs{z}^{2})$ by tangency, $\mathcal{C}_{n,N}=O(\abs{z}^{N+1})$, so at cubic order only $\mathcal{C}_{n,2}$ contributes, consistent with the $O(\Phi^2)$ truncation of $\mathcal{G}_n$.

As in~\eqref{binomial_expansion},
\[
(Du_c)^{\alpha_1}
= (im_c)^{\mathfrak{m}_t(\alpha_1)}
\sum_{p=0}^{\mathfrak{m}_e(\alpha_1)}\sum_{q=0}^{\mathfrak{m}_o(\alpha_1)}
\binom{\mathfrak{m}_e(\alpha_1)}{p}\binom{\mathfrak{m}_o(\alpha_1)}{q}(-1)^q\,
z_1^{|\alpha_1|-p-q}\,\bar{z}_1^{p+q}\,
e^{im_c(|\alpha_1|-2(p+q))x},
\]
and analogously
\[
(Dv_c)^{\alpha_2}
= (im_c)^{\mathfrak{m}_t(\alpha_2)}
\sum_{r=0}^{\mathfrak{m}_e(\alpha_2)} \sum_{s=0}^{\mathfrak{m}_o(\alpha_2)}
\binom{\mathfrak{m}_e(\alpha_2)}{r}\binom{\mathfrak{m}_o(\alpha_2)}{s}(-1)^s\,
z_2^{|\alpha_2|-r-s}\,\bar{z}_2^{r+s}\,
e^{im_c(|\alpha_2|-2(r+s))x}.
\]
The reduced factor $(Du_c)^{\alpha_1-e_j}$ has the same structure with the prefactor $(im_c)^{\mathfrak{m}_t(\alpha_1)-j}$ and with $\mathfrak{m}_e(\alpha_1)$ (even $j$) or $\mathfrak{m}_o(\alpha_1)$ (odd $j$) lowered by one.

\begin{lemma}\label{lem:parity_sign}
$\mathfrak{m}_t(\alpha_1+\alpha_2)\equiv \mathfrak{m}_o(\alpha_1)+\mathfrak{m}_o(\alpha_2)\pmod{2}$.
\end{lemma}
\begin{proof}
Since $\mathfrak{m}_t$ is linear, it suffices to treat each component.
Splitting the defining sum:
\[
\mathfrak{m}_t(\alpha_1)
= \sum_{a\geq 0}2a\,\alpha_{1,2a}
+\sum_{a\geq 0}(2a+1)\,\alpha_{1,2a+1}
\equiv \mathfrak{m}_o(\alpha_1)\pmod{2},
\]
since the first sum is even. Identically $\mathfrak{m}_t(\alpha_2)\equiv \mathfrak{m}_o(\alpha_2)$.
Summing gives the result.
\end{proof}

Substituting $\partial_x^j\Phi_1=\sum_{k}(ik)^j U_k e^{ikx}$ (with the convention $U_{-k}=\overline{U_k}$) into~\eqref{product_rule} and projecting onto $e^{im_cx}$, the integral is nonzero only when
\[
k=m_c\bigl(2(p+q+r+s)-(|\alpha_1|+|\alpha_2|-2)\bigr),
\]
where $0\leq p+q+r+s\leq|\alpha_1|+|\alpha_2|-1$ since $p+q\leq|\alpha_1|-1$, $r+s\leq|\alpha_2|$ for the $\Phi_1$ term, and symmetrically for the $\Phi_2$ term.
For $|\alpha_1|+|\alpha_2|=2$ this reads $k=2m_c(p+q+r+s)\in\{0,2m_c\}$, and $k=0$ is excluded by the mean-zero restriction.
Hence $p+q+r+s=1$, $k=2m_c$, and only the modes $U_{2m_c}$, $V_{2m_c}$ of $\Phi$ enter.

Consider the $\Phi_1$ term of slot $j$, $\alpha_{1,j}\ne0$.
The reduced product $(Du_c)^{\alpha_1-e_j}(Dv_c)^{\alpha_2}$ has total degree one, so the resonant choice $p+q+r+s=1$ selects the single fully conjugated monomial $\bar{z}_1^{\,|\alpha_1|-1}\bar{z}_2^{\,|\alpha_2|}$ with sign $(-1)^{\mathfrak{m}_o(\alpha_1-e_j)+\mathfrak{m}_o(\alpha_2)}=(-1)^{\mathfrak{m}_o(\alpha_1)+\mathfrak{m}_o(\alpha_2)+j}$.
Collecting this with the multiplicity $\alpha_{1,j}$ from~\eqref{product_rule}, the factor $(2im_c)^j$ from $\partial_x^j\Phi_1$ at $k=2m_c$, and the prefactor $(im_c)^{\mathfrak{m}_t(\alpha_1)-j+\mathfrak{m}_t(\alpha_2)}$, the identity $2^j(-1)^j=(-2)^j$ and Lemma~\ref{lem:parity_sign} give the scalar weight
\[
(-im_c)^{\mathfrak{m}_t(\alpha_1+\alpha_2)}\,(-2)^j\,\alpha_{1,j}.
\]
Summing over $j$, multiplying by $2\pi\langle\mathbf{a}_{\alpha_1,\alpha_2},\mathbf{q}_{m_c,n}^*\rangle_{\mathbb{C}^2}$ from the projection, and recalling~\eqref{sigma_n} and~\eqref{ABalpha}, the $\Phi_1$ term contributes exactly $A_{\alpha_1,\alpha_2,n,1}\,\bar{z}_1^{\,|\alpha_1|-1}\bar{z}_2^{\,|\alpha_2|}\,U_{2m_c}$.
Symmetrically, the $\Phi_2$ term contributes $A_{\alpha_1,\alpha_2,n,2}\,\bar{z}_1^{\,|\alpha_1|}\bar{z}_2^{\,|\alpha_2|-1}\,V_{2m_c}$.
The preceding steps yield:
\begin{proposition}\label{prop:cross_quad}
The contribution to $\mathcal{C}_{n,2}$ from the monomial $\mathbf{a}_{\alpha_1,\alpha_2}(Du_c)^{\alpha_1}(Dv_c)^{\alpha_2}$ with $|\alpha_1|+|\alpha_2|=2$, linearised in $(\Phi_1,\Phi_2)$ and projected onto $e_{m_c,n}^*$, is
\[
\begin{cases}
A_{\alpha_1,\alpha_2,n,1}\,\bar{z}_1\,U_{2m_c} & |\alpha_1|=2, \\
A_{\alpha_1,\alpha_2,n,1}\,\bar{z}_2\,U_{2m_c}+A_{\alpha_1,\alpha_2,n,2}\,\bar{z}_1\,V_{2m_c} & |\alpha_1|=|\alpha_2|=1, \\
A_{\alpha_1,\alpha_2,n,2}\,\bar{z}_2\,V_{2m_c} & |\alpha_2|=2,
\end{cases}
\]
with $A_{\alpha_1,\alpha_2,n,j}$ as in~\eqref{ABalpha}.
\end{proposition}
Summing over all monomials with $|\alpha_1|+|\alpha_2|=2$ and collecting by amplitude using~\eqref{CUV} gives
\begin{equation}\label{Ck2}
\mathcal{C}_{n,2}
= \bar{z}_1\bigl(U_{2m_c}C_{U,1,n}+V_{2m_c}C_{V,1,n}\bigr)
+ \bar{z}_2\bigl(U_{2m_c}C_{U,2,n}+V_{2m_c}C_{V,2,n}\bigr).
\end{equation}

\subsection{Center Manifold Coefficients}\label{sec:cm_coefficients}
We determine $U_{2m_c}$ and $V_{2m_c}$ in~\eqref{phi} using the Lyapunov--Perron representation of the center manifold (see~\cite{ptd,csengul2018dynamic}).
Throughout this subsection we evaluate at $\lambda=\lambda_c$ and drop the subscript $\lambda$.
For nearby $\lambda$ the cubic coefficients differ by $O(\abs{\lambda-\lambda_c})$, which is absorbed in the higher-order terms of~\eqref{reduced_system}.

Write $e_n = e_{m_c,n}$, so that $\beta_{m_c,2}=\overline{\beta_1}$, and let $\mathcal{J}$ be the restriction of $L$ to the center subspace, so that
\[
e^{\tau \mathcal{J}} \psi_c
= e^{\tau \beta_1} A e_1 + e^{\tau \overline{\beta_1}} B e_2 + \text{c.c.}
\]
Let $G_2$ denote the symmetric bilinear form associated with the quadratic part of $G$.
By \cite[Theorem~A.1.1]{ptd}, the center manifold satisfies
\begin{equation}\label{comp_proof_1}
\Phi = \int_{-\infty}^0 e^{-\tau L}\,
P_{\text{stable}}\, G_2(e^{\tau \mathcal{J}} \psi_c,\,
e^{\tau \mathcal{J}} \psi_c)\, d\tau + O(\|\psi_c\|^3),
\end{equation}
near the origin.
Expanding the bilinear form, the integrand is a sum over the amplitude monomials $A^2$, $AB$, $B^2$, $A\bar A$, $A\bar B$, $B\bar B$ and their conjugates, the monomial $XY$ carrying the time factor $e^{\tau\mu_{XY}}$ with $\mu_{XY}$ the sum of the corresponding eigenvalues:
\[
\mu_{A^2}=\mu_{A\bar B}=2\beta_1, \qquad
\mu_{B^2}=2\overline{\beta_1}, \qquad
\mu_{AB}=\mu_{A\bar A}=\mu_{B\bar B}=\beta_1+\overline{\beta_1}=0,
\]
using $\Re\beta_1=0$ at criticality.
Since $\Re\beta_{m,n}<0$ for $(m,n) \in \mathfrak{M}_{\mathrm{stable}}$ by~\eqref{PES}, each integral converges there and evaluates to a resolvent, $\int_{-\infty}^0 e^{-\tau(L-\mu_{XY})}\,d\tau=(\mu_{XY}-L)^{-1}$.
Hence
\[
\Phi = \Phi_{AA}A^2+\Phi_{AB}AB+\Phi_{BB}B^2
+\Phi_{A\bar A}A\bar A+\Phi_{A\bar B}A\bar B
+\Phi_{B\bar B}B\bar B+\text{c.c.}+O(\|\psi_c\|^3),
\]
where
\[
\begin{aligned}
\Phi_{AA} &= (2\beta_1-L)^{-1}P_{\text{stable}}G_2(e_1,e_1),\\
\Phi_{A\bar B} &= (2\beta_1-L)^{-1}P_{\text{stable}}
  \bigl(G_2(e_1,\bar e_2)+G_2(\bar e_2,e_1)\bigr),\\
\Phi_{BB} &= (2\bar\beta_1-L)^{-1}P_{\text{stable}}G_2(e_2,e_2),\\
\Phi_{A\bar A} &= (-L)^{-1}P_{\text{stable}}G_2(e_1,\bar e_1),\\
\Phi_{AB} &= (-L)^{-1}P_{\text{stable}}
  \bigl(G_2(e_1,e_2)+G_2(e_2,e_1)\bigr),\\
\Phi_{B\bar B} &= (-L)^{-1}P_{\text{stable}}G_2(e_2,\bar e_2).
\end{aligned}
\]
Since $e_1$ and $e_2$ both carry frequency $m_c$, the products $G_2(e_1,\bar e_1)$, $G_2(e_2,\bar e_2)$, and $G_2(e_1,\bar e_2)+G_2(\bar e_2,e_1)$ each lie in the zero Fourier mode.
By Assumption~\ref{ass:S}\ref{S3} the quadratic part of $G$ and its polarization $G_2$ are total derivatives, so $\Pi_0G_2(f,g)=0$, where $\Pi_0$ is the projection onto the spatial mean.
These products therefore vanish, so $\Phi_{A\bar A}=\Phi_{A\bar B}=\Phi_{B\bar B}=0$.
The same fact keeps the zero-mean space invariant.
It also shows that no zero-harmonic feedback enters the cubic coefficients.

The remaining products lie entirely in the $2m_c$ Fourier mode, with $\mathbb{C}^2$-valued coefficients
\[
\widehat{\mathbf g}_{AA} = \sum_{|\alpha_1|+|\alpha_2|=2}\mathbf{a}_{\alpha_1,\alpha_2}\,
u_{m_c,1}^{|\alpha_1|}\,v_{m_c,1}^{|\alpha_2|}\,(im_c)^{\mathfrak{m}_t(\alpha_1+\alpha_2)},
\quad
\widehat{\mathbf g}_{AB} = \sum_{|\alpha_1|+|\alpha_2|=2}\mathbf{a}_{\alpha_1,\alpha_2}\,
\mathcal{Q}_{\alpha_1,\alpha_2}\,(im_c)^{\mathfrak{m}_t(\alpha_1+\alpha_2)},
\]
with $\mathcal{Q}_{\alpha_1,\alpha_2}$ the polarization~\eqref{Q_AB}, and $\widehat{\mathbf g}_{BB}$ obtained from $\widehat{\mathbf g}_{AA}$ by $(u_{m_c,1},v_{m_c,1})\mapsto(u_{m_c,2},v_{m_c,2})$.
On this block the resolvents act as the matrices $(\mu-M_{2m_c})^{-1}$, well-defined for $\mu\in\{2\beta_1,2\bar\beta_1,0\}$ since $\Re\beta_{2m_c,k}<0$, and the $2m_c$-component of $\Phi$ is
\begin{equation}\label{Phi_coeff_matrix}
\begin{bmatrix}U_{2m_c}\\[2pt]V_{2m_c}\end{bmatrix}
=(2\beta_1 I_2-M_{2m_c})^{-1}\widehat{\mathbf g}_{AA}\,A^2
-M_{2m_c}^{-1}\widehat{\mathbf g}_{AB}\,AB
+(2\bar\beta_1 I_2-M_{2m_c})^{-1}\widehat{\mathbf g}_{BB}\,B^2
+O(\abs{(A,B)}^3).
\end{equation}
We stress that~\eqref{Phi_coeff_matrix} requires no diagonalizability of $M_{2m_c}$.
In particular, the quadratic part of the nonlinearity feeds into the cubic coefficients only through the critical mode $m_c$ and the second harmonic $2m_c$.
No other mode enters at this order.
The $A^2$ and $AB$ coefficient vectors in~\eqref{Phi_coeff_matrix} are exactly $\mathbf h_{AA}$ and $\mathbf h_{AB}$ from~\eqref{eq:Phi_vectors}.
Substituting them directly into~\eqref{Ck2} gives the universal coefficient formula~\eqref{eq:c_hat_matrix}.

When $M_{2m_c}$ has two distinct eigenvalues, expanding $\widehat{\mathbf g}_{AA}$, $\widehat{\mathbf g}_{AB}$, $\widehat{\mathbf g}_{BB}$ in the eigenbasis $\{\mathbf{q}_{2m_c,1},\mathbf{q}_{2m_c,2}\}$ via~\eqref{orth} gives
\begin{equation}\label{UV_2mc}
\begin{aligned}
& U_{2m_c}=\sum_{k=1}^2\Phi_{2m_c,k}\,u_{2m_c,k}+O(\abs{(A,B)}^3),\qquad
V_{2m_c}=\sum_{k=1}^2\Phi_{2m_c,k}\,v_{2m_c,k}+O(\abs{(A,B)}^3),\\
& \Phi_{2m_c,k}=\Phi_{AA,2m_c,k}A^2+\Phi_{AB,2m_c,k}AB+\Phi_{BB,2m_c,k}B^2,
\end{aligned}
\end{equation}
with scalar coefficients, in the notation~\eqref{sigma_n},
\begin{equation}\label{Phi_coeff_derived}
\begin{aligned}
\Phi_{AA,2m_c,n} &= \frac{1}{2\beta_1-\beta_{2m_c,n}}
\sum_{|\alpha_1|+|\alpha_2|=2}\sigma_{2m_c,n}(\alpha_1,\alpha_2)\,
u_{m_c,1}^{|\alpha_1|}\,v_{m_c,1}^{|\alpha_2|},\\
\Phi_{BB,2m_c,n} &= \frac{1}{2\bar\beta_1-\beta_{2m_c,n}}
\sum_{|\alpha_1|+|\alpha_2|=2}\sigma_{2m_c,n}(\alpha_1,\alpha_2)\,
u_{m_c,2}^{|\alpha_1|}\,v_{m_c,2}^{|\alpha_2|},\\
\Phi_{AB,2m_c,n} &= \frac{-1}{\beta_{2m_c,n}}
\sum_{|\alpha_1|+|\alpha_2|=2}\sigma_{2m_c,n}(\alpha_1,\alpha_2)\,
\mathcal{Q}_{\alpha_1,\alpha_2},
\end{aligned}
\end{equation}
in agreement with~\eqref{Phi_coeff}.
The denominators are nonzero since $\Re(2\beta_1)=\Re(2\bar\beta_1)=0$ and $\Re\beta_{2m_c,n}<0$.
The modal sum represented by~\eqref{Phi_coeff_derived} extends regularly to exceptional parameter values where $M_{2m_c}$ fails to be diagonalizable, since~\eqref{Phi_coeff_matrix} remains regular there.
The term $\Phi_{BB,2m_c,k}B^2$ is a quadratic term of the center manifold, not a cubic term of the reduced vector field.
Through~\eqref{Ck2} it feeds resonant cubic terms of the $B$ equation, which we obtain below from the $A$ equation by reflection.

In the distinct-eigenvalue case, substituting \eqref{UV_2mc} into \eqref{Ck2} and passing to $(A,B)$ via \eqref{z12ab} gives
\[
\mathcal{C}_{n,2} = \hat c_{n,1}A|A|^2+\hat c_{n,2}A|B|^2
+\hat c_{n,3}B|A|^2+\hat c_{n,4}B|B|^2
+\hat c_{n,5}\bar A B^2+\hat c_{n,6}A^2\bar B.
\]
For $n=1$ the last four monomials are non-resonant in the $A$ equation and are eliminated by the near-identity transformation.
The resonant contributions $\hat c_{1,1}$, $\hat c_{1,2}$ are given by~\eqref{c_hat_k}, and we compute only these.
For $n=2$ the terms with $\hat c_{2,3}$ and $\hat c_{2,4}$ are resonant in the $B$ equation.
We do not compute them, since the reflection determines them from $\hat c_{1,1}$, $\hat c_{1,2}$ through~\eqref{eq:O2_constraints}.
Combined with the self-interaction of Section~\ref{sec:self-interaction} and the $O(2)$-equivariant reduction of Section~\ref{sec:normal_form}, this yields the normal form~\eqref{reduced_system} with $\zeta = \hat{s}_{1,1}+\hat{c}_{1,1}$ and $\xi = \hat{s}_{1,2}+\hat{c}_{1,2}$ as stated in Theorem~\ref{thm:main}.

\subsection{Normal Form of the Reduced System}\label{sec:normal_form}
The normal form structure is standard \cite{crawford1991symmetry,guckenheimer1983nonlinear,haragus2010local,elphick1987simple,nayfeh2011method}.
The three steps are recorded below.

\emph{Step 1: The Pre-Normal-form Reduced System.} After center manifold reduction, the dynamics on the four-dimensional center space is governed by
\begin{equation}\label{eq:prenf}
\begin{aligned}
\frac{dA}{dt} &= i\omega A + P_1(A,B,\overline{A},\overline{B})
               + \text{h.o.t.}, \\
\frac{dB}{dt} &= -i\omega B + P_2(A,B,\overline{A},\overline{B})
               + \text{h.o.t.},
\end{aligned}
\end{equation}
where $P_1$, $P_2$ are
cubic polynomials in $A,B,\overline{A},\overline{B}$.

\emph{Step 2: Elimination of Non-Resonant Terms.} We seek a near-identity transformation
\begin{equation}\label{eq:nf_transform}
A = Z_1 + h_1(Z_1,Z_2,\overline{Z_1},\overline{Z_2}), \qquad
B = Z_2 + h_2(Z_1,Z_2,\overline{Z_1},\overline{Z_2}),
\end{equation}
with $h_1$, $h_2$ cubic, such that~\eqref{eq:prenf} takes the form
\begin{equation}\label{eq:nf_intermediate}
\begin{aligned}
\frac{dZ_1}{dt} &= i\omega Z_1
  + Z_1(\zeta|Z_1|^2 + \xi|Z_2|^2) + \text{h.o.t.}, \\
\frac{dZ_2}{dt} &= -i\omega Z_2
  + Z_2(\xi'|Z_1|^2 + \zeta'|Z_2|^2) + \text{h.o.t.}
\end{aligned}
\end{equation}
The homological operator of the first component acts diagonally on cubic monomials,
\begin{equation}\label{eq:homological_operator}
\mathcal{L}_{i\omega}\bigl(Z_1^p Z_2^q \overline{Z_1}^{\,r} \overline{Z_2}^{s}\bigr) = i\omega(p - q - r + s - 1)\, Z_1^p Z_2^q \overline{Z_1}^{\,r} \overline{Z_2}^{s},
\end{equation}
so only resonant monomials ($p-q-r+s=1$) survive.
With translation equivariance ($p+q-r-s=1$), the resonant cubic terms are $Z_1|Z_1|^2$ and $Z_1|Z_2|^2$ in $\dot Z_1$.
For the second component the eigenvalue factor is $i\omega(p-q-r+s+1)$, and the resonant cubic terms in $\dot Z_2$ are $Z_2|Z_1|^2$ and $Z_2|Z_2|^2$.
Every resonant monomial has odd degree, so the normal form carries only odd-order terms and the remainder in~\eqref{reduced_system} is $O(\abs{Z}^5)$.
The coefficients $\zeta,\xi,\xi',\zeta'$ are the resonant coefficients of $P_1,P_2$ and are unchanged by the transformation.

\emph{Step 3: $O(2)$ Symmetry Constraints.} Because $L$ and $G$ commute with the $O(2)$ action, the center manifold may be chosen $O(2)$-equivariant \cite[Theorem~3.13]{haragus2010local}.
The reduced field~\eqref{eq:prenf} is then equivariant, and the near-identity transformation preserves this, forcing the normal form~\eqref{eq:nf_intermediate} to be equivariant \cite{golubitsky1985hopf,crawford1991symmetry}.
In the normalization of Lemma~\ref{lem:crit_norm}, the expansion~\eqref{psi_c} of the center part reads
\[
\psi_c = Ae_{m_c,1} + \overline{A}e_{-m_c,1}
       + Be_{m_c,2} + \overline{B}e_{-m_c,2}.
\]
Applying $R$ and using Lemma~\ref{lem:crit_norm} together with $R^2=\mathrm{id}$ gives $R\psi_c = \gamma\overline{B}\,e_{m_c,1} + \gamma\overline{A}\,e_{m_c,2} + \mathrm{c.c.}$, so the reflection acts on the amplitude coordinates as $R:(A,B)\mapsto(\gamma\overline{B},\gamma\overline{A})$, and after the near-identity transformation $(A,B)\to(Z_1,Z_2)$ as
\begin{equation}\label{eq:kappa}
\kappa: (Z_1,Z_2)\mapsto(\gamma\overline{Z_2},\,\gamma\overline{Z_1}).
\end{equation}
Equivariance requires that if $(Z_1,Z_2)$ satisfies~\eqref{eq:nf_intermediate} then so does $(\gamma\overline{Z_2},\gamma\overline{Z_1})$.
Substituting $(Z_1,Z_2)\mapsto(\gamma\overline{Z_2},\gamma\overline{Z_1})$ into~\eqref{eq:nf_intermediate} and conjugating, the unimodular factor $\gamma$ cancels, and comparison with~\eqref{eq:nf_intermediate} yields
\begin{equation}\label{eq:O2_constraints}
\zeta' = \overline{\zeta}, \qquad \xi' = \overline{\xi},
\end{equation}
and substituting~\eqref{eq:O2_constraints} into~\eqref{eq:nf_intermediate} gives the $O(2)$-equivariant normal form~\eqref{reduced_system}.

\section{Concluding Remarks}\label{sec:outlook}

We close by recording several directions in which the present results extend.

The first is the diffusive (sideband) stability of the bifurcated wave families.
Rigorous diffusive-stability theory is available for Turing patterns \cite{mielke1992ginzburg,sukhtayev2018diffusive,wheeler2026convective}.
For the traveling and standing waves the corresponding criteria are formulated through coupled Ginzburg--Landau equations whose cubic coefficients are $\zeta$ and $\xi$, so the closed forms derived here make them evaluable across the polynomial class.
For the stress--strain family~\eqref{system2}, $\xiR$ vanishes identically.
The structure behind this vanishing is the variational form of the nonlinearity \cite{sengul2026variational}.
Whether it survives at quintic order is open.

A second direction is to relax Assumption~\ref{cond:D}\ref{D1}.
A trace of full order $2m_L$ together with the conditions of Assumption~\ref{hopf_assumption} is then not enough for Theorem~\ref{thm:resolvent}, since the inverse Fourier blocks can grow with $\abs{m}$.

A third direction widens the nonlinear class.
The formulas of Theorem~\ref{thm:main} are derived for two-component systems with conservative polynomial nonlinearities.
Their extension to $n$-component systems, to non-conservative $O(2)$-equivariant nonlinearities, and to the quasilinear case $r_G=2m_L$ \cite[Theorem~2.20]{haragus2010local}, remains open.
The two-component structure enters in three places: the reflection $\operatorname{diag}(1,-1)$ in Lemma~\ref{lem:crit_norm}, the $2\times2$ adjugate bound in the proof of Theorem~\ref{thm:resolvent}, and the trace and determinant conditions of Assumption~\ref{hopf_assumption}.
For $n$ components the reflection acts by a signature matrix, the resolvent estimate needs a parameter-elliptic argument for $n\times n$ symbols \cite{denkhieberprus2003}, and the Hopf and gap conditions must be stated through the roots of an $n$-th degree characteristic polynomial.
The matrix formula~\eqref{eq:c_hat_matrix} carries over with $n\times n$ resolvents.
For reaction-diffusion systems with $n\ge3$ components, Villar-Sep\'{u}lveda and Champneys \cite{villarsepulveda2024amplitude} treat this case by multiple-scale asymptotics through fifth order.
An $n$-component version of our framework would recover those coefficients in closed form.

A further direction concerns secondary instabilities of the bifurcating waves.
Their Floquet exponents are expressible through $\zeta$ and $\xi$ \cite[Theorem~2.9]{vangils1986hopf}.
The direction of the secondary bifurcation needs the quintic normal form coefficients \cite[Section~4.1]{crawford1991symmetry}.
The multi-index reduction of Section~\ref{sec:proof} extends to these coefficients and to the degenerate cubic cases \cite{crawford1988degenerate}.

Global polynomiality is likewise inessential: condition~\ref{S1} may be replaced by $G\in C^k(\mathcal Z,\mathcal Y)$, $k\ge5$, with finite-monomial quadratic and cubic jets obeying~\ref{S2}--\ref{S4}, at the cost of assuming $C^k$ regularity rather than inheriting it from polynomiality.
The proof uses polynomiality only through the smoothness of $G$ in Hypothesis~2.1 of Section~\ref{sec:CM_existence} and through the finiteness of these two jets in Sections~\ref{sec:self-interaction}--\ref{sec:cm_coefficients}.
For $3\le k<5$ the formulas for $\zeta$ and $\xi$ still hold, but the remainder in~\eqref{reduced_system} is then only $o(\abs{Z}^3)$.

\section*{Code availability}
A Python implementation of all formulas in Theorem~\ref{thm:main} is provided by the companion package \texttt{o2sym}, available at \url{https://github.com/taylansengul/O2hopf} (release~v0.4.0), with all releases archived at \url{https://doi.org/10.5281/zenodo.20584398}.
It implements the general multi-index machinery of Section~\ref{sec:proof}, checks the assumptions of Section~\ref{sec:main} for a given system, and agrees with the closed forms of Section~\ref{sec:app1} to $10^{-14}$ relative error.
The script \texttt{examples/bilinear\_family.py} regenerates Figure~\ref{fig:app1_sim}.
The code is released under the MIT license.

\appendix

\section{Proofs of the Spectral and Symmetry Preliminaries}\label{app:spectral}

This appendix collects the proofs of the results stated in Section~\ref{sec:main}: Lemma~\ref{lem:syr_G} (admissible monomials), Theorem~\ref{eigenvalue_theorem} (eigenvalues and eigenvectors of $L$), Lemma~\ref{lem:adjoint_eigen} (adjoint eigenvectors) and Lemma~\ref{lem:PES} (spectral configuration at criticality).

\begin{proof}[Proof of Lemma~\ref{lem:syr_G}]
Translation equivariance is automatic since $G$ has constant coefficients.
Under $R$, each factor $\partial_x^j u$ flips sign iff $j$ is odd and each $\partial_x^j v$ iff $j$ is even, so
\[
(Du)^{\alpha_1}(Dv)^{\alpha_2} \,\mapsto\, (-1)^{\mathfrak{m}_o(\alpha_1)+\mathfrak{m}_e(\alpha_2)}\,\bigl[(Du)^{\alpha_1}(Dv)^{\alpha_2}\bigr](-x).
\]
Since $RG=GR$ means $g_1$ is even and $g_2$ is odd under $R$, comparing coefficients monomial by monomial gives~\eqref{syr_G}.
\end{proof}

\begin{proof}[Proof of Theorem~\ref{eigenvalue_theorem}]
Since $\beta_{m,1}\beta_{m,2}=\det M_m\sim b_{1,2m_L}b_{2,2m_L}\,m^{4m_L}\ne0$ by~\eqref{eq:condD} while $\beta_{m,1}+\beta_{m,2}=\tr M_m=O(m^{2m_L})$, we have $|\beta_{m,n}|\to\infty$.
Fix any $\lambda_0\notin\{\beta_{m,n}\}$.
Then
\[
\det(\lambda_0I_2-M_m)=\lambda_0^2-\lambda_0\tr M_m+\det M_m
\sim b_{1,2m_L}b_{2,2m_L}m^{4m_L},
\]
whereas the entries of $\operatorname{adj}(\lambda_0I_2-M_m)$ are $O(m^{2m_L})$.
Hence $(\lambda_0I_2-M_m)^{-1}=O(m^{-2m_L})$ as $|m|\to\infty$.
The finitely many remaining blocks are invertible by the choice of $\lambda_0$, so the blockwise inverse is uniformly bounded and $\lambda_0\in\rho(L)$.
Moreover, the resolvent maps $\mathcal H$ boundedly into $\mathcal H^{2m_L}$ and is compact after composition with the compact embedding $\mathcal H^{2m_L}\hookrightarrow\mathcal H$.
Hence $\sigma(L)$ consists of eigenvalues only \cite[Thm.~III.6.29]{kato1966perturbation}, and every eigenvalue is some $\beta_{m,n}$ because the Fourier blocks are invariant.

By~\eqref{Mm}, $\tr M_m$ and $\det M_m$ are real and even in $m$, so the characteristic equation~\eqref{char} is the same for $\pm m$ and has real coefficients.
Its roots are therefore both real or a conjugate pair, and labeling them so that $\beta_{m,n}=\overline{\beta_{-m,n}}$ gives the two cases.
\end{proof}

\begin{proof}[Proof of Lemma~\ref{lem:adjoint_eigen}]
Fix a block with $\beta_{m,1}\ne\beta_{m,2}$.
Since $\sigma(L^*)=\overline{\sigma(L)}$, we label the adjoint eigenpairs by $L^*e^*_{m,n}=\overline{\beta_{m,n}}\,e^*_{m,n}$, which on the Fourier block reads $M_m^*\,\mathbf q^*_{m,n}=\overline{\beta_{m,n}}\,\mathbf q^*_{m,n}$.
Conjugating~\eqref{Mm} flips the sign of exactly the odd powers of $m$, so $\overline{M_m}=M_{-m}$ and hence $M_m^*=\overline{M_m}^{\,T}=M_{-m}^T$.
Since $\overline{\beta_{m,n}}=\beta_{-m,n}$ by Theorem~\ref{eigenvalue_theorem}, this is~\eqref{Mm_adjoint_eigen}.
For~\eqref{orth}, modes with $m\neq m'$ are orthogonal in $L^2$.
For $m=m'$, $n\neq n'$, $\beta_{m,n}\langle\mathbf q_{m,n},\mathbf q^*_{m,n'}\rangle=\langle M_m\mathbf q_{m,n},\mathbf q^*_{m,n'}\rangle=\langle\mathbf q_{m,n},M_m^*\mathbf q^*_{m,n'}\rangle=\beta_{m,n'}\langle\mathbf q_{m,n},\mathbf q^*_{m,n'}\rangle$, so the pairing vanishes when $\beta_{m,n}\neq\beta_{m,n'}$.
The same simplicity makes $\langle\mathbf q_{m,n},\mathbf q^*_{m,n}\rangle\neq0$, so the normalization $\langle\mathbf q_{m,n},\mathbf q^*_{m,n}\rangle=1/2\pi$ can be imposed.
\end{proof}

\begin{proof}[Proof of Lemma~\ref{lem:PES}]
Conditions~(i) and~(ii) give $\tr M_{m_c}(\lambda_c) = 0$ and $\det M_{m_c}(\lambda_c) > 0$, so the roots of the characteristic equation~\eqref{char}, labeled so that $\Im\beta_{m_c,1}>0$, are
\[
\beta_{m_c,1}(\lambda_c) = i\omega, \qquad \beta_{m_c,2}(\lambda_c) = -i\omega,
\qquad \omega = \sqrt{\det M_{m_c}(\lambda_c)} > 0,
\]
proving the first two lines of~\eqref{PES}.
Since $\tr^2 M_{m_c} - 4\det M_{m_c} < 0$ at $\lambda_c$ by~(i)--(ii), the roots of~\eqref{char} at $m=m_c$ remain a simple non-real conjugate pair for $\lambda$ near $\lambda_c$.
They are therefore smooth in $\lambda$, with $\beta_{m_c,2}(\lambda) = \overline{\beta_{m_c,1}(\lambda)}$.
Differentiating $\beta_{m_c,1}(\lambda) + \beta_{m_c,2}(\lambda) = \tr M_{m_c}(\lambda)$ then gives $2\,\frac{d}{d\lambda}\Re\beta_{m_c,1}(\lambda_c) = \frac{d}{d\lambda}\tr M_{m_c}(\lambda_c)\ne 0$ by~(iii).
Since $\tr M_m$ and $\det M_m$ depend only on $m^2$, the same holds for $-m_c$, establishing the third line.

For the last line, fix $(m,n)\in\mathfrak{M}_{\mathrm{stable}}$.
Then $m\ne\pm m_c$, so $\tr M_m(\lambda_c)<0$ and $\det M_m(\lambda_c)>0$ by~(iv) and~(ii), and both roots of~\eqref{char} have negative real part (if real, both are negative since their sum is negative and their product positive, and if complex, their common real part is $\tfrac12\tr M_m<0$). Hence $\Re\beta_{m,n}(\lambda_c) < 0$ for all $(m,n)\in\mathfrak{M}_{\mathrm{stable}}$.
By~\eqref{Mm}, $\tr M_m(\lambda_c)$ is a polynomial in $m^2$.
It is nonconstant, since a constant value would equal $\tr M_{m_c}(\lambda_c)=0$ by~(i), contradicting~(iv).
Condition~(iv) then forces its leading nonvanishing coefficient to be strictly negative, so $\tr M_m(\lambda_c) \to -\infty$ as $|m|\to\infty$.
By Assumption~\ref{cond:D}\ref{D1} the off-diagonal entries of $M_m$ have degree at most $2m_L-1$ in $m$, so $\det M_m(\lambda_c)=b_{1,2m_L}b_{2,2m_L}\,m^{4m_L}+O(m^{4m_L-2})$ and $|\tr M_m(\lambda_c)|=O(m^{2m_L})$.
Condition~(ii) forces $b_{1,2m_L}b_{2,2m_L}>0$, so $\det M_m(\lambda_c)/|\tr M_m(\lambda_c)|\to\infty$ as $|m|\to\infty$.
Hence there exist $\eta > 0$ and $m_0 \ge 1$ such that $\det M_m(\lambda_c) \ge \eta\,|\tr M_m(\lambda_c)|$ and $|\tr M_m(\lambda_c)| \ge 4\eta$ for all $|m|\ge m_0$.
Fix $|m|\ge m_0$ and write $\tau := \tr M_m(\lambda_c)$, $\delta := \det M_m(\lambda_c)$.

If $\tau^2 - 4\delta < 0$, the eigenvalues are complex conjugates with $\Re\beta_{m,n} = \tau/2 \le -2\eta$.

If $\tau^2 - 4\delta \ge 0$, the larger root is $\beta_+ = (\tau + \sqrt{\tau^2 - 4\delta})/2$.
From $\delta \ge \eta|\tau|$ and $\sqrt{1-x} \le 1 - x/2$ for $x\in[0,1]$,
\[
\sqrt{\tau^2 - 4\delta} \le |\tau|\sqrt{1 - 4\eta/|\tau|} \le |\tau| - 2\eta,
\]
so $\beta_+ \le -\eta$.

In either case, $\Re\beta_{m,n}(\lambda_c) \le -\eta$ for all $(m,n)\in\mathfrak{M}_{\mathrm{stable}}$ with $|m|\ge m_0$.
For the finitely many remaining stable modes, $\Re\beta_{m,n}(\lambda_c) < 0$ by the negativity step.
Taking the supremum,
\[
\sup_{(m,n)\in\mathfrak{M}_{\mathrm{stable}}}\Re\beta_{m,n}(\lambda_c) < 0.
\]
The first and last lines locate $\Re\beta_{m,n}(\lambda_c) = 0$ precisely on $\mathfrak{M}_{\mathrm{center}}$, so the center subspace is exactly four-dimensional.
\end{proof}

\section*{Statements and Declarations}

\subsection*{Competing interests}
The authors declare that they have no competing interests.

\subsection*{Funding}
The authors did not receive support from any organization for the submitted work.

\subsection*{Author contributions}
All authors contributed to the study conception, the analysis, and the writing of the manuscript, and all authors read and approved the final manuscript.

\subsection*{Use of generative AI and AI-assisted technologies}
During the preparation of this work the authors used Claude (Anthropic) and ChatGPT (OpenAI) to edit and improve the wording of the manuscript, and to assist with the implementation and testing of the companion \texttt{o2sym} package.
After using these tools, the authors reviewed and edited the content as needed and take full responsibility for the content of the published article.
All mathematical results, proofs and formulas are the authors' own and were verified independently of them.

\bibliographystyle{plain}
\bibliography{o2hopf}

\end{document}